\documentclass[12pt,reqno,hidelinks]{amsart}
\usepackage{amssymb}
\usepackage{amsmath, amssymb, verbatim, url, mathrsfs} 
\usepackage{stmaryrd}
\usepackage{dutchcal}
\usepackage{tabularx}
\usepackage{mathtools}
\usepackage{verbatim}
\usepackage{amsthm}
\usepackage{framed}
\usepackage{cite}
\usepackage{wasysym}
\usepackage{upgreek}
\usepackage{color}
\usepackage[dvipsnames]{xcolor}
\usepackage{array}
\usepackage{tensor}
\usepackage{accents} 
\usepackage[T1]{fontenc}
\usepackage{dsfont}
\usepackage[colorlinks,linkcolor=blue,citecolor=blue]{hyperref}
\usepackage{enumerate}
\usepackage{enumitem}
\usepackage[normalem]{ulem}
\usepackage{longtable}
\usepackage{mathtools}
\usepackage{subcaption}  
\usepackage{ulem}

\numberwithin{equation}{section}

\theoremstyle{definition} 

\newtheorem{proposition}{Proposition}[section]
\newtheorem{lemma}[proposition]{Lemma}
\newtheorem{corollary}{Corollary}[section]
\newtheorem{theorem}{Theorem}[section]

\newtheorem{remark}{Remark}[section]

\newtheorem*{theorem*}{Theorem}
\newtheorem*{mquestion*}{Main Question}
\newtheorem*{claim*}{Claim}

\newtheorem*{intuition*}{Intuition}

\newcommand{\vertiiii}[1]{{\left\vert\kern-0.25ex\left\vert\kern-0.25ex\left\vert\kern-0.25ex\left\vert #1 \right\vert\kern-0.25ex\right\vert\kern-0.25ex\right\vert\kern-0.25ex\right\vert}}
\newcommand{\vertiii}[1]{{\left\vert\kern-0.25ex\left\vert\kern-0.25ex\left\vert #1 \right\vert\kern-0.25ex\right\vert\kern-0.25ex\right\vert}}

\newcommand{\Rbb}{\mathbb{R}}

\newcommand{\AND}{{\quad\text{and}\quad}}

\newcommand{\p}[1]{
	\begin{pmatrix}
		#1
	\end{pmatrix}
}

\newcommand{\be}{\begin{equation}}
	\newcommand{\ee}{\end{equation}}

\allowdisplaybreaks

\begin{document}

	 \title[Singularity-free solutions and  scalarization in nonlinear ESGB cosmology]{Proofs on singularity-free solutions and  scalarization in nonlinear Einstein-scalar-Gauss-Bonnet cosmology}
	
	\author{Chihang He, Chao Liu* and Jinhua Wang}

	\address[Chihang He]{School of Mathematics and Statistics, Huazhong University of Science and Technology, Wuhan 430074, Hubei Province, China.}
	\email{hechihang@foxmail.com}

	\address[Chao Liu]{Center for Mathematical Sciences and School of Mathematics and Statistics, Huazhong University of Science and Technology, Wuhan 430074, Hubei Province, China.}
	\email{chao.liu.math@foxmail.com}

	\address[Jinhua Wang]{School of Mathematical Sciences, Xiamen University,
		Xiamen 361005, Fujian Province, China. }
	\email{wangjinhua@xmu.edu.cn}

	\begin{abstract} 
 
  The search for singularity-free cosmological solutions has become a highly active topic in the physics community in recent years, yet existing results are largely numerical or based on asymptotic analysis. To place these developments on a firm mathematical footing, we rigorously establish the global existence and estimates of a class of singularity-free cosmological solutions to the \textit{fully nonlinear Einstein--scalar system} in Einstein-scalar-Gauss-Bonnet gravity with quadratic coupling, providing proofs of previous numerical results in mathematical perspective. We further prove nonlinear scalarization triggered by a Gauss-Bonnet-induced tachyonic instability. Our analysis relies on a novel structural identity, the power identity, which yields decoupled differential inequalities for the Hubble parameter. This framework provides a new method for converting numerical evidence into a mathmatical proof for nonlinear systems.

        \vspace{2mm}

   
	\end{abstract}

	\maketitle

\section{Introduction}
In this paper, we focus on the global existence of solutions to the following \textit{Einstein-scalar system} in Einstein-scalar-Gauss-Bonnet (ESGB) gravity
\begin{equation*}
	 R_{\mu\nu}-\frac{1}{2}Rg_{\mu\nu}+\Gamma_{\mu\nu}=2\partial_{\mu}\phi\partial_{\nu}\phi-(\nabla\phi)^{2}g_{\mu\nu} \quad \text{(Einstein equation)},
\end{equation*}
\begin{equation*}
	\frac{1}{\sqrt{-g}}\partial_{\mu}\left[\sqrt{-g}\partial^{\mu}\phi\right]-\frac{1}{4}f'(\phi) R_{\mathrm{GB}}^{2}=0 \quad \text{(scalar equation)} ,
\end{equation*}
where $g$ is the metric, $R$ and $R_{\mu \nu}$ denote the scalar curvature and Ricci tensor, respectively, $\phi$ is the scalar field, $f$ is the coupling function, and $R_{\mathrm{GB}}^2$ is the Gauss-Bonnet term, defined by
\begin{equation}\label{e:RGB}
	R_{\mathrm{GB}}^2 = R^2 - 4R_{\mu\nu}R^{\mu\nu} + R_{\mu\nu\rho\sigma}R^{\mu\nu\rho\sigma},
\end{equation}
and $\Gamma_{\mu\nu}$ is given by
\begin{equation}\label{e:Gm}
\begin{aligned}
\Gamma_{\mu\nu} = 2R\nabla_{\mu}\nabla_{\nu}f &+ 4(R_{\mu\nu}-\frac{1}{2}Rg_{\mu\nu})\nabla^{\alpha}\nabla_{\alpha}f - 8R_{\alpha(\mu}\nabla^{\alpha}\nabla_{\nu)}f \\
&\quad + 4R^{\alpha\beta}g_{\mu\nu}\nabla_{\alpha}\nabla_{\beta}f - 4R_{ \ \ \mu\alpha\nu}^{\beta}\nabla^{\alpha}\nabla_{\beta}f .
\end{aligned}
\end{equation}

Spacetimes in classical general relativity generically develop singularities, as indicated by the Hawking-Penrose singularity theorems \cite{HawkingPenrose1965,HawkingPenrose1970,Hawking2010}. Such singularities illustrate the breakdown of general relativity and motivate the search for singularity-free solutions \cite{Hawking2005}, a problem that has been investigated extensively \cite{Starobinsky1980,Trodden1993,Brandenberger1993,Ellis2003,Hartle1983}.
Among the proposed extensions of general relativity, the ESGB theory has attracted particular attention due to its nontrivial curvature-scalar coupling and second-order field equations \cite{Zwiebach1985,GrossSloan1987}. Substantial progress has been made in understanding ESGB cosmologies through numerical simulations and perturbative analyses \cite{Antoniadis1994,Rizos1994,Kanti1999,Kanti2015,Kanti2015a,Kawai1999,Sberna2017,Sberna2017a,Hikmawan2016,Kawai1998,Easther1996}. Concurrently, the phenomenon of scalarization, originally discovered by Damour and Esposito-Farèse \cite{Damour1993,Damour1996}, has drawn significant attention in various contexts, including studies of neutron stars, black holes and cosmology within scalar-tensor, Horndeski, and ESGB theories themselves \cite{Doneva2024,Anson2019,Silva2018,Doneva2018,Andreou2019,Bakopoulos2019,Doneva2018a,Antoniou2018,Antoniou2018a,Astefanesei2020,Minamitsuji2019,BlazquezSalcedo2018,Dima2020,Herdeiro2021,Berti2021,ZH.STEFANOV2008,Herdeiro2018,Cardoso2020,Cardoso2013,Myung2018}. For a detailed account of these previous studies, see \S~\ref{sec:related work} for a detailed review. 

However, despite these advances, \textit{to the best of our knowledge, a rigorous mathematical proof is still lacking for both the global dynamics of ESGB cosmologies and the scalarization mechanism within this specific setting.} 

From a mathematical perspective, constructing singularity-free solutions in nonlinear modified gravity models is highly challenging. Due to the strong nonlinearity and complexity of the problem, in the current article, we restrict our analysis to the case of spatially homogeneous and isotropic spacetimes, under which the full system reduces to a coupled nonlinear system of ordinary differential equations (ODE) in $(H, \phi, \dot{\phi})$.  Even as an ODE system, it remains highly nonlinear, with curvature-dependent terms, making it difficult to apply standard methods for proving global existence or for decoupling the evolution of $H$ from that of the scalar field.

\textbf{Main contributions.}
The present work provides a rigorous analytical treatment of the \textit{nonlinear Einstein-scalar system in the ESGB gravity} in mathematical perspective. Under a quadratic Gauss-Bonnet coupling and flat spatial geometry (see the companion paper \cite{he2025} for the exponential coupling case), we establish global existence, obtain sharp and explicit bounds for all cosmological variables, and mathematically provide an analysis of a nonlinear  scalarization \footnote{The term ``scalarization'' originates in the context of compact objects such as black holes and neutron stars. It has since been extended to cosmology, where it refers to the development of nontrivial scalar field configurations triggered by a tachyonic instability (see \cite{Doneva2024} and \cite[\S VI]{Anson2019}). In this work, we adopt this latter notion of scalarization in the cosmological background.} in this setting. Our contributions can be summarized as follows:

\begin{enumerate}[leftmargin=*]
    \item \textbf{New analytic framework for the Einstein-scalar system.} To overcome the mathematical difficulties posed by the nonlinear coupling, we find a special structure of the Einstein-scalar system referred to as the \emph{power identity}, obtained by combining the Friedmann constraint with the scalar field equation. This identity plays a role of a nonlinear conserved relation and yields decoupled differential inequalities for the Hubble parameter. \textit{Notably, the power identity which plays a central role in our analysis is not limited to the ESGB model. As a methodological approach, it potentially offers a broader framework (see Remark \ref{R:pi}) for developing analytical proof techniques for a wide class of differential equations.}

    \item \textbf{Rigorous construction of singularity-free ESGB cosmologies.}
    For an explicit family of initial data, we obtain a unique global homogeneous and isotropic solution, and derive upper and lower bounds for $H$, $\phi$, and $\dot{\phi}$. These bounds quantitatively match the numerically observed behaviors reported in \cite{Rizos1994, Sberna2017a}, providing mathematically support for these long-standing numerical predictions.  

    \item \textbf{Rigorous nonlinear proof of scalarization under ESGB setting.}
    For an enlarged class of initial data, we show that the scalar field necessarily undergoes curvature-induced tachyonic growth, resulting in a future global spacetime with a nontrivial self-increasing scalar profile. Our analysis yields explicit growth estimates and identifies the precise dynamical mechanism through which scalarization emerges.
\end{enumerate}

\subsection{Main Theorem}
\label{main theorem}
Let us begin with a brief review of the fundamentals of ESGB gravity. 
We consider a string-inspired gravitational theory in which a scalar field $\phi$ is coupled to gravity through a coupling function $f(\phi)$ multiplying the quadratic Gauss-Bonnet term, in addition to the Ricci scalar. The theory is defined by the following action (see, e.g., \cite{Rizos1994,Kanti1999,Kanti2015,Kanti2015a,Kawai1999,Sberna2017a,Hikmawan2016})
\begin{equation}\label{e:S}
	S_{\text{ESGB}} = \frac{1}{16\pi} \int d^4x \sqrt{-g} \left( R - 2\partial_\mu\phi\partial^\mu\phi - V_\phi - \lambda f(\phi) R^2_{\mathrm{GB}} \right),
\end{equation}
where  $V_\phi$ is the scalar field potential and $\lambda$ represents the coupling constant

In this article, following the cosmological framework adopted in \cite{Rizos1994,Kanti1999,Kanti2015,Kanti2015a,Kawai1999,Sberna2017a,Hikmawan2016}, we consider a free scalar field, a positive coupling constant, and a quadratic coupling function\footnote{See the companion article \cite{he2025} for the case of exponential coupling, which involves additional difficulties. }, specified by
\begin{equation}\label{e:para}
	V_\phi=0,\quad \lambda=1 \AND f(\phi) = \frac{1}{2}\phi^2.
\end{equation}

Taking the variation of this action \eqref{e:S} with respect to both the metric tensor and the scalar field yields the corresponding \textit{Einstein-scalar field equations} in ESGB theory, 
\begin{equation}
    R_{\mu\nu}-\frac{1}{2}Rg_{\mu\nu}+\Gamma_{\mu\nu}=2\partial_{\mu}\phi\partial_{\nu}\phi-(\nabla\phi)^{2}g_{\mu\nu},
    \label{eq:2.4++++++}
\end{equation}
and
\begin{equation}
    \frac{1}{\sqrt{-g}}\partial_{\mu}\left[\sqrt{-g}\partial^{\mu}\phi\right]-\frac{1}{4}\phi R_{\mathrm{GB}}^{2}=0,
    \label{eq:2.3++++++}
\end{equation}
where $R^2_{\mathrm{GB}}$ and $\Gamma_{\mu\nu}$ are defined by \eqref{e:RGB} and \eqref{e:Gm}. 

We note that the system \eqref{eq:2.4++++++}–\eqref{eq:2.3++++++}, or equivalently the action \eqref{e:S}–\eqref{e:para}, is invariant under the discrete transformation $\phi \rightarrow -\phi$, reflecting a $\mathbb{Z}_2$ symmetry (see, e.g.,  \cite[Chap.~4.4]{Sberna2017a}).

In addition, we assume a homogeneous and isotropic universe described by the Friedmann-Lemaître-Robertson-Walker (FLRW) metric, 
\begin{equation}\label{e:FLRW}
	g(t) = -dt^2 + a^2(t) \left(  dr^2 + r^2(d\theta^2 + \sin^2\theta \, d\varphi^2) \right),
    \end{equation}
  where $a(t):=a_0 e^{\int^t_0H(s)ds}$ is the scale factor, $H=\dot{a}/a$ denotes the Hubble parameter, and the dot represents a derivative with respect to time.

For the FLRW metric \eqref{e:FLRW}, the Einstein-scalar field equations \eqref{eq:2.4++++++}--\eqref{eq:2.3++++++} reduce to the following explicit form (see \S\ref{sec:2} for details)
\begin{align}
	&3H^2 = \dot{\phi}^2 + 12H^3 \phi \dot{\phi}, \quad (\text{Friedmann equation; Hamiltonian constraint}),
	\label{eq:2.5} \\ 
	& 2\dot{H}    + 3 H^2 = -\dot{\phi}^2 + 8(\dot{H} + H^2) H \phi \dot{\phi} + 4H^2 \big(\dot{\phi}^2 + \phi \ddot{\phi} \big),
	\label{eq:2.4} \\
	&\ddot{\phi} + 3\dot{\phi}H + 6\phi(H^2 + \dot{H})H^2   = 0, \quad (\text{scalar field
     equation}).
	\label{eq:2.6}
\end{align} 
This article focuses on the analysis of this nonlinear ODE system.

\begin{theorem}[Global Existence and Bounds] \label{theorem:2.1}
Suppose the initial data
\begin{equation}\label{initial-data}
	(a_0,\beta,\alpha) := (a,H,\phi)|_{t=0} 
\end{equation}
satisfy
	\begin{gather}
		\alpha = 0, \quad a_0>0, \label{eq:2.12} \\
		\beta  \in \biggl(0, \frac{\sqrt{3}}{3}\biggr).  \label{eq:2.13}
	\end{gather} 
	Then, by imposing the symmetry-breaking initial condition\footnote{As shown in \S \ref{sec:2}, the genuinely free data for the system \eqref{eq:2.5}--\eqref{eq:2.6} are exactly given by \eqref{initial-data}.  In particular,  due to the constraint \eqref{eq:2.5}, $\dot \phi(0)$ is completely determined by $\alpha$ and $\beta$ up to a sign: $\dot\phi(0)=-6\alpha \beta^3 \pm  \sqrt{\left(6\alpha \beta^3\right)^2 + 3 \beta^2}$.  If we write $\dot{\phi}(0) = -6\alpha \beta^3 + s \sqrt{\left(6\alpha \beta^3\right)^2 + 3 \beta^2}$ with $s = \pm 1$, then the system \eqref{eq:2.4++++++}–\eqref{eq:2.3++++++}, or equivalently the action \eqref{e:S}–\eqref{e:para}, admits a symmetry under simultaneous sign reversal of both $s$ and $\phi$. This symmetry has been discussed in several references; see, e.g., \cite[Chap.~4.4]{Sberna2017a}. In the present theorem, since $\alpha = 0$, imposing the condition \eqref{eq:2.11} explicitly breaks this symmetry.} 
    \begin{equation}
		\dot{\phi}(0) > 0 , \label{eq:2.11}
    \end{equation} 
 there exists a unique \textit{globally singularity-free,  homogeneous, and isotropic  FLRW solution} $(g,\phi)\in C^2( (-\infty, +\infty))$,  where $g$ is defined by \eqref{e:FLRW}, solving the ESGB field equations  \eqref{eq:2.4++++++}--\eqref{eq:2.3++++++} with initial data \eqref{initial-data}.  
    Moreover, this FLRW solution, described by $\phi$ and $H$, satisfies the following estimates
	\begin{enumerate}[leftmargin=*]
	\item (Scalar field bounds, see Fig. \ref{fig:bound3})  The scalar field $\phi$ satisfies
	\begin{enumerate}[leftmargin=*]
		\item For $t \in (-\infty,0)$, 
			\begin{equation*} 
			\frac{\sqrt{3}}{12}(1-e^{-12t})  < \phi(t) < \mathfrak{H}(t) , 
			\label{eq:2.15}
		\end{equation*}
        where $\mathfrak{H}(t)$ is defined by
        \begin{equation*}
            \mathfrak{H}(t):=\begin{cases}
			\displaystyle	\frac{\sqrt{3}}{6\beta^2}(1-e^{-6\beta^3t}) , \quad &\text{if }0< \beta \leq \sqrt{\frac{5}{27}} , \\
			\displaystyle	\frac{25\sqrt{3}}{24}(1-e^{-\frac{48}{125}t}) , \quad &\text{if }\sqrt{\frac{5}{27}}< \beta < \frac{\sqrt{3}}{3}  .
		\end{cases} 
        \end{equation*}
		\item For $t\in (0,+\infty)$, 	
        \begin{equation*}
			\frac{-1+\sqrt{1+\frac{24\beta^2}{5}\ln{(1+5\beta t)}}}{4\sqrt{3}\beta^2} < \phi(t) < \sqrt{3} \ln(\beta t+1).  
			\label{eq:2.13++}
		\end{equation*} 
	\end{enumerate}
	\item (Hubble parameter bounds, see Fig. \ref{fig:bound1} and \ref{fig:bound2})
The Hubble parameter $H$  satisfies the following bounds. 
	\begin{enumerate}[leftmargin=*]
	\item For $t \in (-\infty,0)$, 	
    \begin{equation*}
		\mathfrak{L}(t) <H(t) <\frac{2 \beta +\sqrt{2}-\left(\sqrt{2}-2 \beta \right) e^{3 \sqrt{2} t}}{\sqrt{2} \left(2 \beta  +\sqrt{2}+\left(\sqrt{2}-2 \beta \right) e^{3 \sqrt{2} t}\right)},  
		\label{eq:2.16}
	\end{equation*}
    where $\mathfrak{L}(t)$ is defined by
\begin{equation*}
    \mathfrak{L}(t):=\begin{cases}
		\frac{227 \beta }{2 e^{\frac{454 \beta  t}{45}}+225} >\beta , \quad &\text{if }0< \beta \leq \sqrt{\frac{5}{27}} , \\
		\frac{2 \beta }{5 \beta - (5 \beta -2) e^{4 t} } >\frac{2}{5} , \quad &\text{if }\sqrt{\frac{5}{27}}< \beta <\frac{\sqrt{3}}{3}  .
	\end{cases}
\end{equation*}
	\item For $ t \in (0,+\infty)$, 
	\begin{equation*}
	\frac{1}{5t+\frac{1}{\beta}} < H(t) < \frac{1}{t+\frac{1}{\beta}}.
	\label{eq:2.16++}
\end{equation*}

	\end{enumerate}
    \item (Bounds on $\dot{\phi}$)
    The time derivative of the scalar field $\dot{\phi}$ satisfies  
    \begin{enumerate}[leftmargin=*]
    \item For $t \in (-\infty,0)$,
\begin{equation*}
    \mathfrak{W}(t)<\dot{\phi}<\sqrt{3}e^{-12t},\quad \text{for all } t\in (\mathcal{T}_-,0),
\end{equation*}
where $\mathfrak{W}(t)$ is defined by \begin{equation*}
            \mathfrak{W}(t):=\begin{cases}
			\displaystyle	\sqrt{3}\beta e^{-6\beta^3t} , \quad &\text{if }0< \beta \leq \sqrt{\frac{5}{27}} , \\
			\displaystyle	\frac{2\sqrt{3}}{5}e^{-\frac{48}{125}t} , \quad &\text{if }\sqrt{\frac{5}{27}}< \beta < \frac{\sqrt{3}}{3}  .
		\end{cases} 
        \end{equation*}
	\item For $ t \in (0,+\infty)$, 
\begin{equation*}
   \frac{1}{\sqrt{1+12\beta^2\ln{(\beta t+1)}}}\frac{\sqrt{3}}{5t+\frac{1}{\beta}}<\dot{\phi}<\frac{\sqrt{3}}{t+\frac{1}{\beta}}. 
\end{equation*}
\end{enumerate} 
\item (Scalar factor bounds) The scalar factor $a$ satisfies the following bounds
\begin{enumerate}
    \item For $t \in (-\infty,0)$,
    \begin{equation*}
a_0 \left( \frac{2}{(\sqrt{2}\beta + 1) + (1 - \sqrt{2}\beta) e^{3\sqrt{2} t}} \right)^{\frac{1}{3}} e^ { \frac{\sqrt{2}}{2}t }<a(t)< a_0 \mathfrak{G}(t),
    \end{equation*}
    where $\mathfrak{G}(t)$ is defined by
    \begin{equation*}
        \mathfrak{G}(t)= \begin{cases} 
\left( \dfrac{227}{225 e^ { -\frac{454 \beta t}{45}}  + 2} \right)^{\frac{1}{10}},  \quad &\text{if }0< \beta \leq \sqrt{\frac{5}{27}}, \\
          \biggl( \dfrac{2}{5\beta e^{-4t}-(5\beta-2)}\biggr)^{\frac{1}{10}}
		    , \quad &\text{if }\sqrt{\frac{5}{27}}< \beta <\frac{\sqrt{3}}{3} .
\end{cases}
    \end{equation*}
    \item For $t \in (0,+\infty)$,
\begin{equation*}
    a_0 \left( 5\beta t + 1 \right)^{\frac{1}{5}}<a(t)< a_0 (\beta t + 1).
\end{equation*}

\end{enumerate} 
	\end{enumerate}  
\end{theorem}

\begin{figure}[!htbp]
    \centering
    \begin{subfigure}[b]{0.45\textwidth}
        \includegraphics[width=\linewidth]{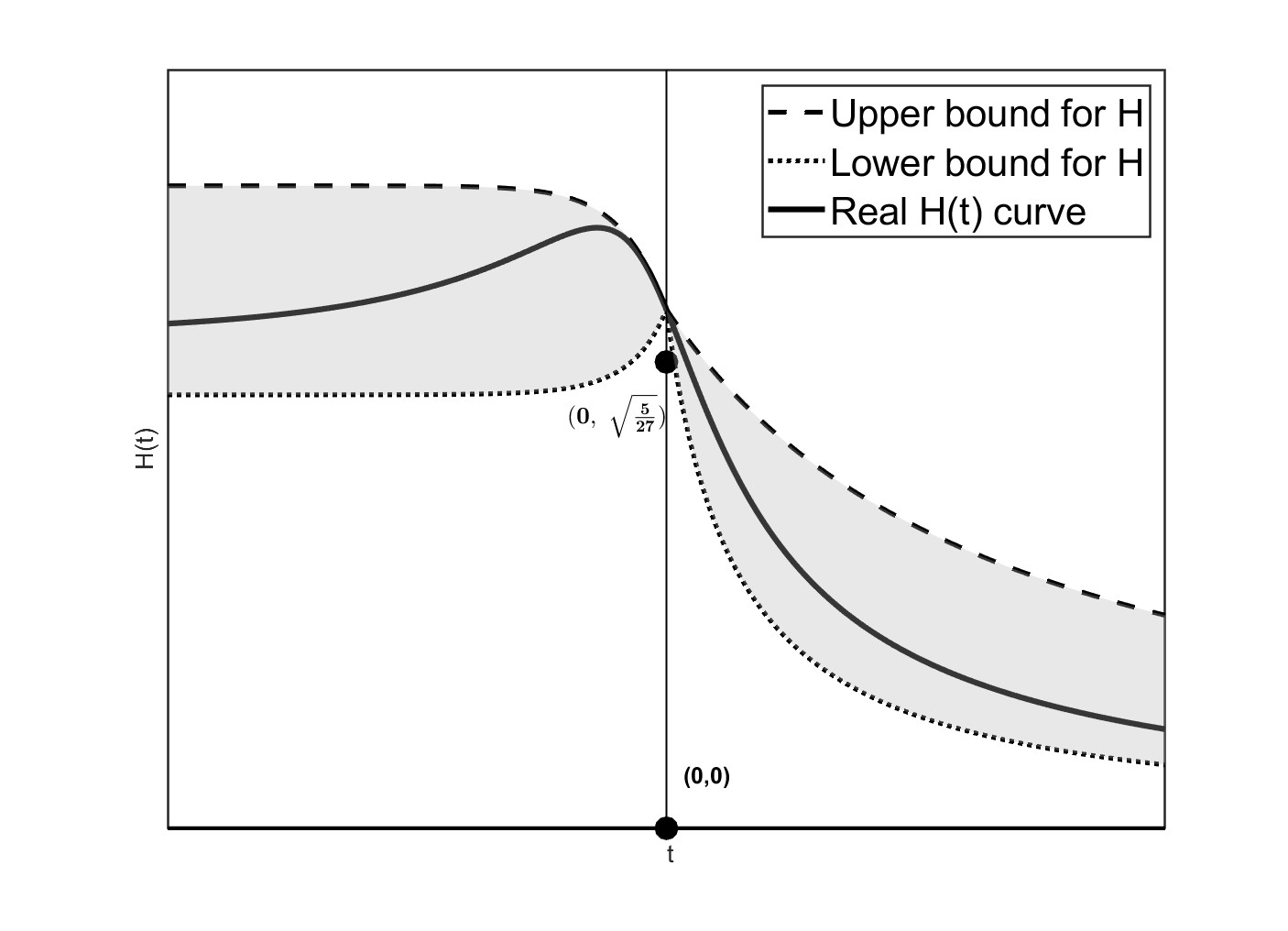}
        \caption{Bounds for $H$ ($\beta> \sqrt{5/27}$)} \label{fig:bound1}
    \end{subfigure}
    \begin{subfigure}[b]{0.45\textwidth}
        \includegraphics[width=\linewidth]{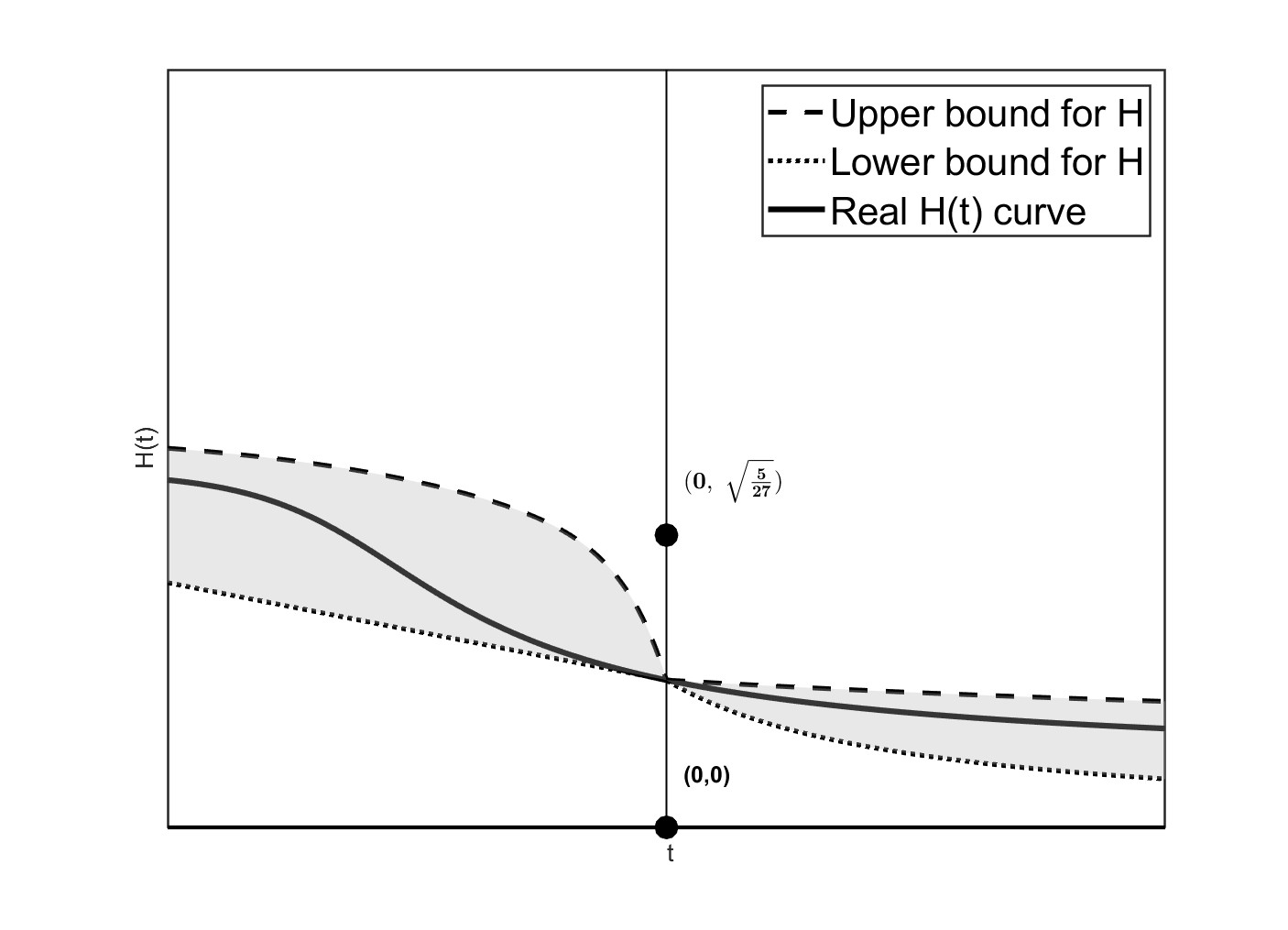}
        \caption{Bounds for $H$  ($\beta\leq \sqrt{5/27}$)}   \label{fig:bound2}
    \end{subfigure}
    \\
   \centering  
    \begin{subfigure}[b]{0.45\textwidth}
        \includegraphics[width=\linewidth]{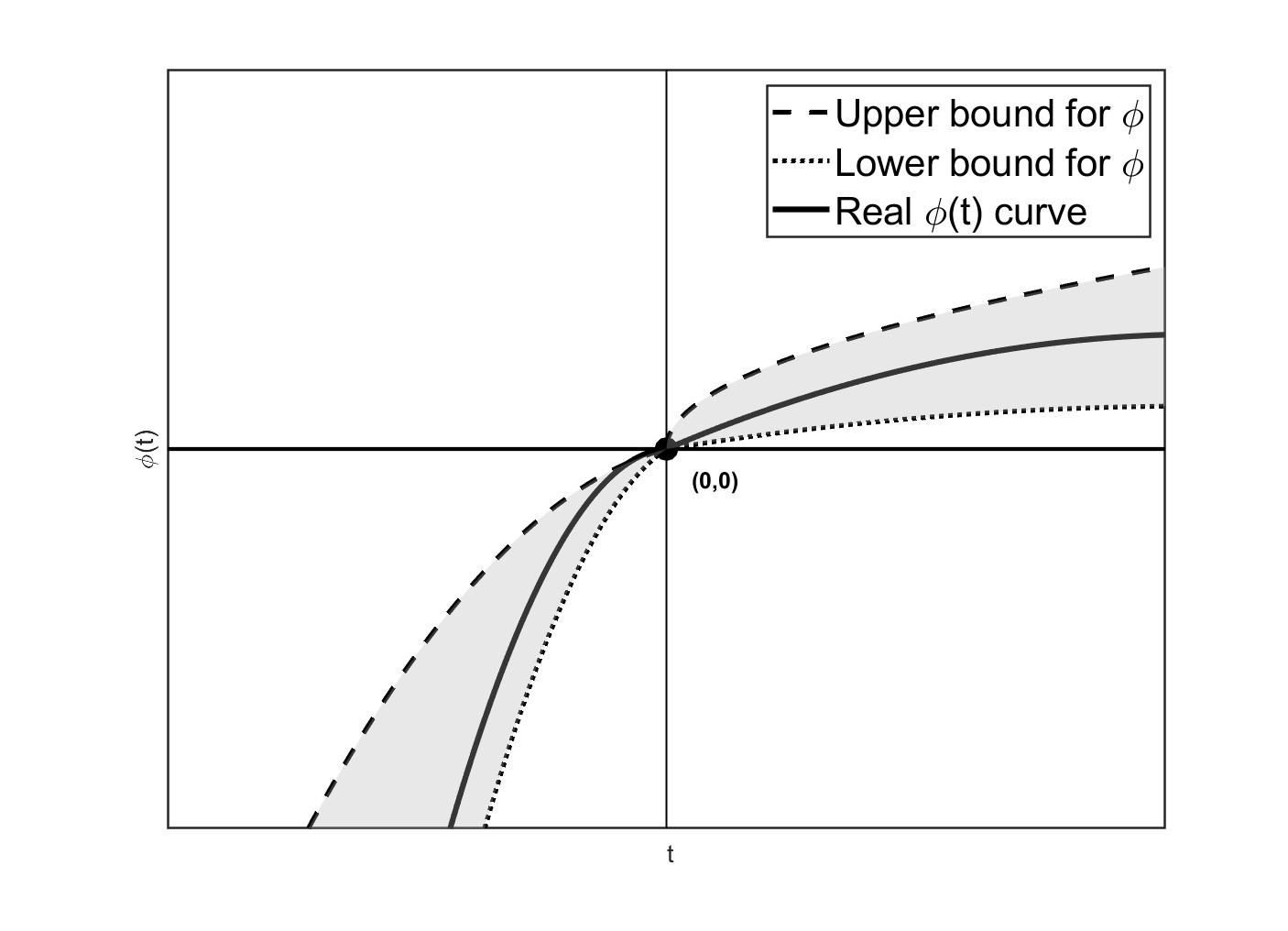}
        \caption{Bounds for $\phi$}     \label{fig:bound3}
    \end{subfigure}
    \caption{Bounds for $H$ and $\phi$}
    \label{fig:bound}
\end{figure}

In fact, if we restrict our attention to the future time direction ($t > 0$) and consider a much broader class of initial data, the above theorem still holds, with only slight modifications to the estimates for $\phi$. This is formalized in the following theorem, which also implies nonlinear scalarization triggered by a tachyonic instability. As such, it can be viewed as a toy model for scalarization in ESGB cosmology.
To describe this broader class of initial data, we introduce the following quantities
\begin{equation}
	\kappa:=\kappa(\alpha,\beta)=\frac{-4 \beta^3 \alpha \gamma 
		- \gamma^2 
		+ 4 \beta^2 \gamma^2 
		- 24 \beta^6 \alpha^2 \
		- 3 \beta^2}{2 - 8 \beta \alpha \gamma + 24 \beta^4 \alpha^2}, 
	\label{eq:2.10++++}
\end{equation}
where
\begin{equation}
	\gamma:=- 6 \alpha  \beta^3 + \sqrt{\left( 6\alpha \beta^3 \right)^2 + 3 \beta^2}. 
	\label{eq:2.9++++}
\end{equation}

\begin{theorem}[Scalarization Theorem]
	\label{theorem:1.2} 
	Suppose the initial data $(a_0,\beta,\alpha)$, defined by \eqref{initial-data}, lie in the following admissible data set
	\begin{equation}
			\mathcal{A}:= \left\{(a_0,\beta,\alpha) \in  (0,+\infty) \times \left(0, \tfrac{\sqrt{3}}{3}\right) \times [0,+\infty)\;\bigg|\;-5\beta^2 < \kappa(\alpha,\beta) < -\beta^2 \right\}, 
		\label{eq:1.13++++}
	\end{equation} 
	and satisfies the condition
	\begin{equation}
		\dot\phi(0) > 0.
		\label{1.14++++}
	\end{equation}
	Then there exists a unique \textit{global,  homogeneous, and isotropic  FLRW solution} $(g, \phi) \in C^2([0, +\infty))$,  where $g$ is given by \eqref{e:FLRW}, that solves the ESGB field equations \eqref{eq:2.4++++++}--\eqref{eq:2.3++++++}. This solution exhibits scalarization in the sense that the nontrivial scalar field $\phi$ grows as
	\begin{equation*}
		\frac{-1+\sqrt{1+\frac{48\beta^2}{5}\ln(1+5\beta t)}}{4\sqrt{3}\beta^2} < \phi(t) < \sqrt{3} \ln(\beta t+1)+\alpha,
	\end{equation*}
	 the Hubble parameter $H$ satisfies the following bounds:
	\begin{equation*}
		\frac{1}{5t+\frac{1}{\beta}} < H(t) < \frac{1}{t+\frac{1}{\beta}},
	\end{equation*} 
    and the time derivative of the scalar field $\dot{\phi}$ satisfies  
    \begin{equation*}
    \frac{1}{\sqrt{1+12\beta^2\bigl(\ln{(\beta t+1)+\alpha\bigr)}}}\frac{\sqrt{3}}{5t+\frac{1}{\beta}}<\dot{\phi}<\frac{\sqrt{3}}{t+\frac{1}{\beta}}  ,
\end{equation*} 
	for all $t \in (0, +\infty)$.
\end{theorem}

\begin{remark}
The global solution described in Theorem~\ref{theorem:2.1} is obtained by solving both the forward and backward Cauchy problems. Specifically, given initial data at \( t = 0 \), we establish the global existence of the solution in both the future (\( t > 0 \)) and the past (\( t < 0 \)) time directions.
\end{remark}

\begin{remark}
The initial conditions do not lose generality due to the following aspects:
\begin{enumerate}
    \item The Hamiltonian constraint determines $\dot{\phi}(0)$ only up to a sign, and the choice $\dot{\phi}(0)>0$ merely selects one of the two solution branches. The other choice can be solved similarly.
    \item In fact, our result implies that any $\alpha\neq 0$ can be normalized to $0$ by a time shift. Hence, without loss of generality, we set $\alpha=0$ for simplicity.
    \item Although numerical evidence suggests a larger region of initial data, we, for simplify, give a restriction on $\beta$, that is $\beta\in(0,\frac{\sqrt{3}}{3})$. This restriction is possible to extend to a larger region but it will cost more difficult calculations and constructions.
\end{enumerate}
\end{remark}

\begin{remark}
While establishing global-in-time existence for this ODE system ensures that the variables do not blow up, the strict criterion for the absence of a physical singularity requires examining both the behavior of curvature invariants and geodesic completeness.
\begin{enumerate}
    \item Regarding curvature invariants, Theorem \ref{theorem:2.1} establishes global bounds for $a$, $H$, and $\dot{H}$ (where \(\dot{H}\) can be controlled by functions of \(H\); see \ref{3} in \S \ref{S:outline}). Since curvature invariants are smooth functions of these variables, they remain finite for all $t$.
    \item As for geodesic completeness, the future-directed causal geodesics can be directly proved to be complete by using the method and tools of \cite{borde2003inflationary,Ruth2006} with the bounds derived in Theorem \ref{theorem:2.1}. In the past direction, however, analogous to \cite[\S II]{borde2003inflationary}, it can be shown that past-directed null geodesics and non-comoving past-directed timelike geodesics (while comoving geodesics remain complete) have finite affine parameter and are therefore incomplete.
    \item  In fact, the spacetime given in Theorem \ref{theorem:2.1} is past-asymptotic to a de Sitter flat patch. It features an extendable boundary analogous to the coordinate boundary of a de Sitter flat patch, even though  the scalar field $\phi$ and $\dot\phi$ diverge in the limit $t \to -\infty$.  We only restrict our attention to the region where $\phi$ and $\dot\phi$ remain finite and the results are consistent with the early numerical simulation in \cite{Rizos1994, Sberna2017} on the interval $t \in (-\infty, +\infty)$.
\end{enumerate}
\end{remark}

\begin{remark}[Non-emptiness of the admissible set $\mathcal{A}$]
	We \textit{claim} that the admissible set $\mathcal{A}$ is non-empty. This is demonstrated by the following two examples. 
	\begin{enumerate}[leftmargin=*]
		\item Let $\alpha = 0$ and $(a_0, \beta) \in (0, +\infty) \times \bigl(0, \tfrac{\sqrt{3}}{3}\bigr)$, which is consistent with the initial data in Theorem~\ref{theorem:2.1}. In this case, one can verify that $-5\beta^2 < \kappa = 6\beta^4 - 3\beta^2 < -\beta^2$. Thus, the data of Theorem~\ref{theorem:1.2} extend that of Theorem~\ref{theorem:2.1} to a broader class of initial data.
		
		\item Let $\alpha = 1$, $\beta = \tfrac{1}{2}$, and $a_0 > 0$. A direct computation yields $\kappa = \tfrac{5\sqrt{21} + 24}{68} \approx -0.6899$, which satisfies $-\tfrac{5}{4} < \kappa < -\tfrac{1}{4}$. Hence, this choice of parameters also lies in $\mathcal{A}$.
	\end{enumerate}
\end{remark}

\begin{remark}
	The estimates of Theorem \ref{theorem:2.1} depicted in Fig. \ref{fig:bound} are consistent with the numerical results of $\phi$ and $H$ presented in Fig. \ref{fig:simulation}. These numerical findings have also been given in \cite[Fig.~4.4]{Sberna2017a}, while the corresponding phase diagram can be found in \cite{Rizos1994}. 
 
\begin{figure}[!htbp]
	\centering
	\begin{subfigure}[b]{0.48\textwidth}
		\includegraphics[width=\linewidth]{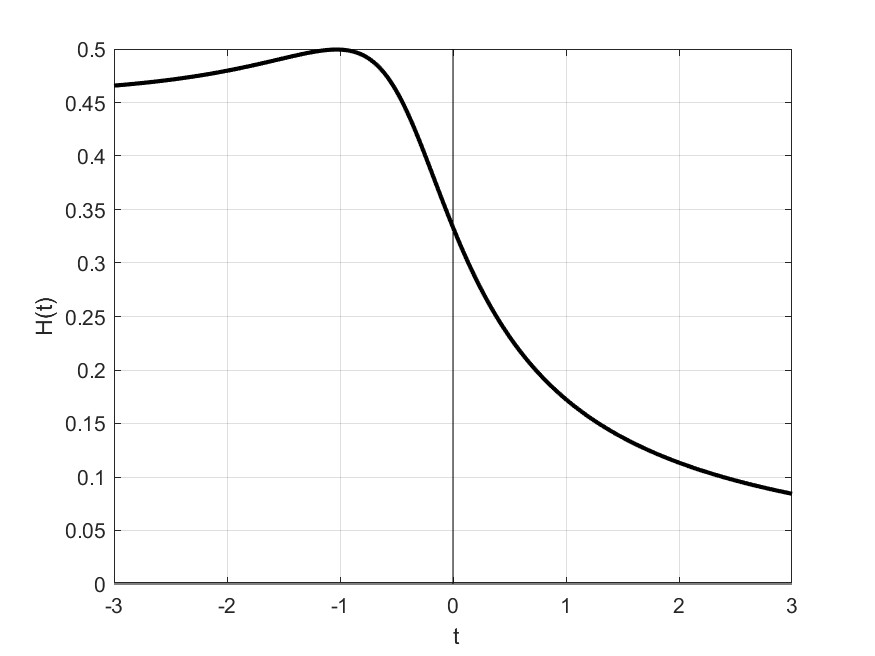}
		\caption{Evolution of $H$}
	\end{subfigure}
	\hfill
	\begin{subfigure}[b]{0.48\textwidth}
		\includegraphics[width=\linewidth]{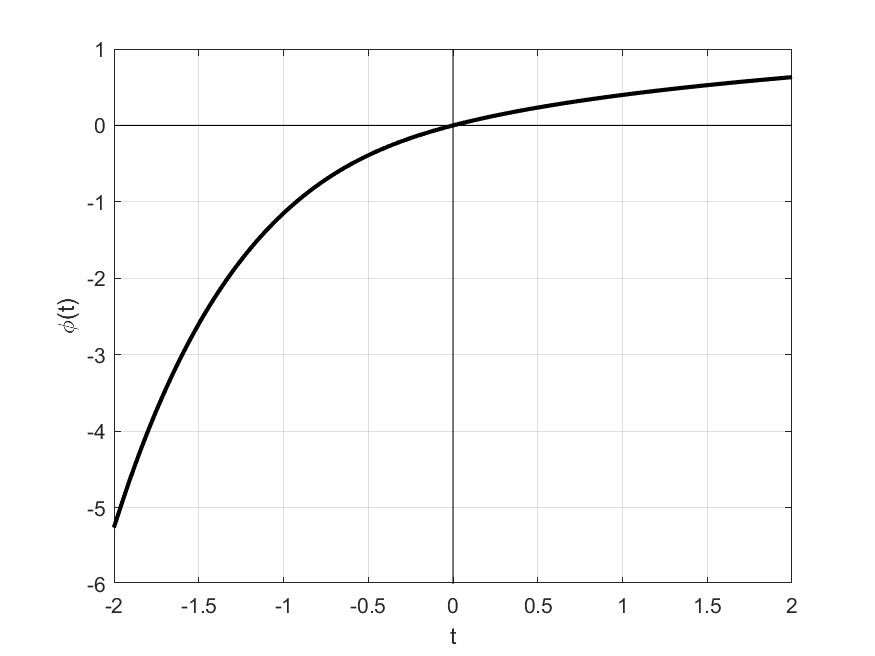}
		\caption{Evolution of $\phi$}
	\end{subfigure}
	\caption{
		 Using initial values  
		$\alpha=0$, $\beta=\frac{1}{3}$, these plots manifest the key features of the system: \\
		 (a) Solutions exist  
		 for all $t \in (-\infty, +\infty)$,\\
		 (b) The Hubble parameter  $H$ remains positive and decreases monotonically for $t>0$,\\
		 (c) The scalar field increases monotonically and crosses zero at $t=0$.}
	\label{fig:simulation}
\end{figure} 
\end{remark}

\subsection{Related Works} \label{sec:related work}
Singularity-free cosmological solutions have been extensively studied in the physics literature, notably in the works of Starobinsky, Mukhanov, Trodden, Brandenberger, and Hawking and Ellis \cite{Starobinsky1980,Trodden1993,Brandenberger1993,Ellis2003,Hartle1983}.

Einstein-scalar-Gauss-Bonnet theory occupies a distinguished place in the modified gravity theories. A series of works in the physics literature \cite{Antoniadis1994,Rizos1994,Kanti1999,Kanti2015,Kanti2015a,Kawai1999,Sberna2017,Sberna2017a,Hikmawan2016,Kawai1998,Easther1996} have shown that, through asymptotic analysis, linearized arguments, and extensive numerical computations, for some certain choices of coupling functions, ESGB cosmology may exhibit \emph{singularity-free expanding solutions}, as well as \emph{tachyonic instabilities} leading to scalarization, which is originally identified by Damour and Esposito-Far\`ese in scalar-tensor theories \cite{Damour1993,Damour1996}. These studies strongly suggest that the Gauss-Bonnet coupling can suppress singularity formation and dynamically generate nontrivial scalar profiles.

Specificly, Kanti, Rizos  and Tamvakis \cite{Kanti1999} extended the work of Rizos and Tamvakis \cite{Rizos1994} to include cases with arbitrary spatial curvature. These works established the existence of singularity-free scalar-Gauss-Bonnet cosmologies for a broad class of coupling functions, and considered non-flat FLRW geometries as well, while the novelty of the present work is mathematical in nature: we provide a global-existence proof and explicit bounds for a restricted quadratic-coupling, spatially flat FLRW case.

Furthermore, in \cite{Kanti2015a,Kanti2015}, they proposed that in early-time Gauss–Bonnet gravity, the contribution of the Ricci scalar is negligible compared to the dominant Gauss–Bonnet term. This simplification allows for analytical solutions that effectively approximate the full ESGB theory during the early-time regime. Linearized perturbation analyses were carried out in \cite{Sberna2017,Kawai1998,Kawai1999}. 

The linear tachyonic instability is also referred to as the Jeans instability in different physical contexts. At the linear level, both instabilities originate from the same underlying mechanism: an effective negative mass squared in the perturbation equations. However, their nonlinear evolutions are governed by the underlying physical models, leading to distinct behaviors due to the different nonlinear terms.  A series of works on the nonlinear Jeans instability have been conducted by one of the authors (see \cite{Liu2022,Liu2022b,Liu2023a,Liu2023b,Liu2024}).

The exponential coupling case was analyzed in the companion paper \cite{he2025}. The lack of decoupling introduces new difficulties in deriving bounds for $\dot{H}$.

\subsection{Outlines and Overview of Methods}\label{S:outline}
 
The proof unfolds in three main steps: establishing local existence (\S\ref{sec:2}), deriving  bounds via \textit{a key structural identity} and \textit{decoupled differential inequalities for $H$} (\S\ref{sec:3}), and extending the solution globally (\S\ref{sec:4}).

The core technical ingredient lies in a structural identity, referred to as the \textit{power identity}, derived from the derivatives of the constraint and scalar field equations. This nonlinear relation, quadratic in $\dot\phi$ and $\phi$ with coefficients depending on $\dot H$ and $H$, governs the evolution of $H$ and $\phi$.

Using a contradiction argument, we leverage the power identity to extract \textit{a set of decoupled differential inequalities for $H$}, which yield upper and lower bounds of $H$.  These bounds then enable a hierarchical derivation of estimates for $\phi$, $\dot\phi$, and $a$.

In line with this approach, we now sketch the proof. Since \S\ref{sec:2} establishes local existence by reformulating the system \eqref{eq:2.5}--\eqref{eq:2.6} as an ODE $\frac{d}{dt} \mathcal{U} = \mathcal{F}(\mathcal{U})$ and applying the local existence theorem, and \S\ref{sec:4} extends these solutions globally via continuation arguments, we focus here on the estimates for $H$ and $\phi$ developed in \S\ref{sec:2}.

The heart of the analysis lies in \S\ref{sec:3}, which develops a hierarchy of estimates 
of $H(t)$ and $\phi(t)$ on the intervals $t \in (\mathcal{T}_-, 0)$ and $t \in (0, \mathcal{T}_+)$.  The foundation of all estimates is the quadratic identity in $ \dot{\phi} $ and $ \phi $  from
	Proposition \ref{proposition:3.2++} (the power identity), 
	\begin{equation*}
\mathcal{P}:=H\left(\biggl(2H^2-1-\frac{\dot{H}}{3H^2}\biggr)\dot{\phi}^2 - 8H^3\phi \dot{\phi} \right.  
\left. - 12H^6\biggl(1+\frac{\dot{H}}{H^2}\biggr)\phi^2\right) = 0 .
	\end{equation*}
	This identity  roughly serves as the energy conservation law. The \textit{key idea} for handling the coupled system \eqref{eq:2.5}--\eqref{eq:2.6} is to use the power identity to decouple $H$ by a set of  differential inequalities of $H$. 
In specific,  the key strategy involves
	\begin{enumerate}[leftmargin=*]
		\item  Defining auxiliary functions $B_\ell$, $\ell=1,\cdots,5$ (e.g., $ B_1 = \dot{H} + 5H^2 $) that capture the dominant behaviors  in the evolution of $H$.
		
		\item Proving that the evolution of $H$ satisfies,  for instance, $B_\ell>0$ using proof by contradiction.   \textit{Assuming the opposite} of the desired bound (e.g., $B_\ell \leq 0$ when proving $B_\ell > 0$) and showing inconsistency with $\mathcal{P} = 0$.
		
		\item Solving decoupled differential inequalities of $H$ (e.g., $B_1 = \dot{H} + 5H^2 > 0$) to derive explicit bounds (e.g.,  $H(t) > \frac{1}{5t + \frac{1}{\beta}}$).\label{3}
	\end{enumerate}
	
The overall structure of the argument is as follows: 
 Local existence $\rightarrow$  Lower bounds for $H$ $\rightarrow$  Upper bounds for $H$ $\rightarrow$  Improved Lower bounds for $H$ $\rightarrow$ Bounds for $\phi$ $\rightarrow$  Global existence and estimates on $\dot{\phi}$ and $a$.

\begin{remark}\label{R:pi}
The ``power identity'' employed in this work is a structural consequence of the Friedmann and scalar field equations rather than a feature specific to ESGB theory. In homogeneous cosmologies with additional fluid components, the field equations couple with the cosmological Euler equations, resulting in a more intricate nonlinear ODE system with additional variables. While the power identity strategy remains applicable in principle, constructing suitable bounding differential inequalities in these cases becomes technically more demanding. We will proceed this in the near future (see \S \ref{S:F}).
As a novel analytical tool, the power identity method shows potential beyond the present setting, though its full scope remains unclear. Our preliminary explorations highlight two primary challenges: first, constructing an effective structure identity (analogue to the power identity in this article) is challenging, which typically stems from the equations or conservation laws; second, the associated bounding equations must admit valid solutions, whose forms are highly problem-dependent. For instance, in systems involving exponential coupling \cite{he2025}, the coupled structure prevents a clean decoupling, yet the method remains tractable. 
Furthermore, we are extending this framework to static, spherically symmetric black holes with nontrivial scalar fields. This application, while significantly more complex, underscores the method's inherent generality. A systematic understanding of its full range of applicability is left for future work.
\end{remark}

\subsection{Future Works}\label{S:F}
While this work focuses on flat FLRW cosmology with quadratic coupling, the power identity method is generalizable. It has already been successfully applied to exponential coupling \cite{he2025}, and further extensions are underway.
\begin{enumerate}
\item Incorporating matter, spatial curvature, or scalar potentials into the FLRW background.
\item Providing a rigorous proof of black hole scalarization in (extended) scalar–tensor theories consistent with the numerical results \cite{Doneva2018,Silva2018,Doneva2024}.
\item Attempting to apply the method to general nonlinear ordinary or partial differential equations.
\end{enumerate}
These topics are currently in progress.

\section{Constraints and local existence}
\label{sec:2}
We consider the ESGB field equations \eqref{eq:2.4++++++}–\eqref{eq:2.3++++++} under the assumption of a homogeneous and isotropic universe, where the metric is given by the FLRW form \eqref{e:FLRW}.  In this setting, the scalar curvature and the Gauss–Bonnet term reduce to
\begin{equation*}
	R = 6 \left( \frac{\ddot{a}}{a} + \left( \frac{\dot{a}}{a} \right)^2 \right) \AND R^2_{GB} = 24 \left( \frac{\dot{a}}{a} \right)^2 \frac{\ddot{a}}{a}. 
\end{equation*} 
The system \eqref{eq:2.4++++++}–\eqref{eq:2.3++++++} then reduces to the following equations for the Hubble parameter \[H := \frac{\dot{a}}{a}\] and the scalar field $\phi$
\begin{equation}
	3H^2 = \dot{\phi}^2 + 12H^3 \phi \dot{\phi}, \quad 
	\label{eq:2.5++}
\end{equation}
\begin{equation}
	2\dot{H}    + 3 H^2 = -\dot{\phi}^2 + 8(\dot{H} + H^2) H \phi \dot{\phi} + 4H^2 \big(\dot{\phi}^2 + \phi \ddot{\phi} \big),
	\label{eq:2.4+}
\end{equation}
\begin{equation}
	\ddot{\phi} =- 3\dot{\phi}H - 6\phi(H^2 + \dot{H})H^2  . 
	\label{eq:2.6++}
\end{equation}
We point out that \eqref{eq:2.5++} is known as the Friedmann equation, or  the Hamiltonian constraint.

\subsection{Initial data and the constraint}\label{s:data} 
Analogous to the Cauchy problem for Einstein–scalar systems (cf.~\cite{ChoquetBruhat2009,Ringstroem2009}),  \eqref{eq:2.5++} serves as the \emph{Hamiltonian constraint}, restricting the admissible initial data $(a,H,\phi,\dot{\phi})|_{t=0}$. If this constraint holds at $t=0$, then it remains valid for all $t$,  since by virtue of \eqref{eq:2.4+}–\eqref{eq:2.6++}, we have
\begin{equation}\label{e:dtctnt}
\partial_t(3H^2-\dot{\phi}^2 - 12H^3 \phi \dot{\phi}) = 0.
\end{equation}
Solving \eqref{eq:2.5++} for $\dot{\phi}$ yields  
	\begin{equation}\label{e:dotphi}
		\dot{\phi}  =   -6\phi H^3 +(-1)^\iota \sqrt{\left(6\phi H^3\right)^2 + 3H^2}, \quad  (\iota=0,1 )  . 
	\end{equation}

Therefore, specifying the initial data $(a_0,\beta,\alpha) := (a, H, \phi)|_{t=0}$ determines $\dot{\phi}|_{t=0}$ via \eqref{e:dotphi}, that is, 
\begin{equation}
	\dot{\phi} |_{t=0} 
	=  -6\alpha \beta^3 +(-1)^\iota \sqrt{\left(6\alpha \beta^3\right)^2 + 3 \beta^2}, \quad  (\iota=0,1 ).
	\label{eq:2.7}
\end{equation}
This gives two admissible initial data sets:
\begin{align}\label{e:dtset}
	  (a,H,\phi,\dot{\phi})|_{t=0} =\begin{cases}
	  (a_0,\beta,\alpha,-6\alpha \beta^3 + \sqrt{\left(6\alpha \beta^3\right)^2 + 3 \beta^2}), \\
	  (a_0,\beta,\alpha,-6\alpha \beta^3 - \sqrt{\left(6\alpha \beta^3\right)^2 + 3 \beta^2}).
	  \end{cases}
\end{align}

\subsection{Local existences}
We begin with a useful estimate
\begin{lemma}\label{t:Dvalue}
	If $\dot{\phi}$ satisfies \eqref{e:dotphi} on some set $\mathcal{I}$, then for all $t \in \mathcal{I}$,
	\begin{equation*}
	2 - 8H \phi\dot{\phi} + 24H^4 \phi^2>0  . 
\end{equation*}  
\end{lemma}

\begin{proof}
	(1) If $H = 0$ or $\phi = 0$ at some $t \in \mathcal{I}$, then the expression reduces to $2 - 8H \phi\dot{\phi} + 24H^4 \phi^2=2 > 0$ at those points.
	
	(2) If $H \neq 0$ and $\phi \neq 0$ at $t \in \mathcal{I}$, then using \eqref{e:dotphi}, we have
	\begin{equation*}
	H\phi\dot{\phi} = -6\phi^2 H^4 \pm \sqrt{\left(6\phi^2 H^4\right)^2 + 3\phi^2H^4}.
\end{equation*} 
	Applying \eqref{eq:B2}, it follows that
	\begin{align*}
	H\phi\dot{\phi}  = & -6\phi^2 H^4 \pm \sqrt{\left(6\phi^2 H^4\right)^2 + 3\phi^2H^4} \notag \\
    \leq  & -6\phi^2 H^4+ \sqrt{\left(6\phi^2 H^4\right)^2 + 3\phi^2H^4}<\sqrt{3}H^2|\phi| . 
\end{align*}
Substituting this into the expression yields
	\begin{equation*}
	2 - 8H \phi\dot{\phi} + 24H^4 \phi^2>	 2(1-2\sqrt{3}H^2|\phi|)^2 \geq 0 . 
\end{equation*}
	
	This completes the proof.
\end{proof}

\begin{proposition}[Local Existence]\label{t:locext}
	Given initial data $(a_0, \beta, \alpha) := (a,H,\phi)|_{t=0}$, there exists a constant $T > 0$, such that the system \eqref{eq:2.4+}--\eqref{eq:2.6++}  with this initial data  admits a pair of   solutions $(a,H,\phi,\dot{\phi}) \in C^1((-T,T),\mathbb{R}^3)$ corresponding to the initial data sets in \eqref{e:dtset}. 
\end{proposition}
\begin{proof} 
Substitute \eqref{eq:2.6++} into \eqref{eq:2.4+}. Using Lemma \ref{t:Dvalue}, we obtain
	\begin{equation}\label{eq-dotH}
	\dot{H} = \frac{-4H^3 \phi\dot{\phi} - \dot{\phi}^2 + 4H^2 \dot{\phi}^2 - 24H^6 \phi^2 - 3H^2}{2 - 8H \phi\dot{\phi} + 24H^4 \phi^2}.
\end{equation}  
	Let $
		\Phi:=\dot{\phi}$. 
	Then gathering $H=\dot{a}/a$, \eqref{eq-dotH}, $\dot{\phi} = \Phi$ and \eqref{eq:2.6++} together, the system becomes
	\begin{equation}\label{e:locsys} 
		 \frac{d}{dt} \p{a\\H\\\phi\\\Phi}= \p{ a H\\	F_1(H,\phi,\Phi) \\
		  	\Phi \\ F_2(H,\phi,\Phi) },
	\end{equation}
	where
	\begin{align}
		F_1(H,\phi,\Phi):=& \frac{-4H^3 \phi \Phi- \Phi^2 + 4H^2 \Phi^2 - 24H^6 \phi^2 - 3H^2}{2 - 8H \phi \Phi + 24H^4 \phi^2},  \label{e:F1} \\
		F_2(H,\phi,\Phi):=& - 3\Phi H - 6H^2 \phi \left(H^2 + F_1(H,\phi,\Phi) \right). \label{e:F2} 
	\end{align}
		We can verify that  $F_1, F_2\in C^1(\Rbb^3)$. By \S\ref{s:data}, the initial data set \eqref{e:dtset} provides two admissible choices.  Applying Theorem~\ref{theorem:PL}, we obtain a unique solution $(a,H, \phi, \Phi) \in C^1((-T, T), \mathbb{R}^3)$ to the system \eqref{e:locsys} for each choice of initial data.  Hence, the original system \eqref{eq:2.4+}–\eqref{eq:2.6++} admits a pair of $C^1$ solutions $(a,H, \phi, \dot{\phi})$ on $(-T, T)$ corresponding to the initial data $(a_0,\beta,\alpha)$. This completes the proof. 
\end{proof}

\section{Estimates of the FLRW Solution in the ESGB System}
\label{sec:3}

In the previous section, we have proven the local existence of a FLRW solution $(a, H, \phi, \dot{\phi}) \in C^1((-T, T), \mathbb{R}^3)$ to the ESGB system \eqref{eq:2.4++++++}--\eqref{eq:2.3++++++}, given initial data $(a_0, \beta, \alpha)$. In this section, we will derive estimates for this solution within its interval of existence.

\subsection{Estimates for $H$}

To proceed, we define $\mathcal{T}_- \in [-\infty, 0)$ and $\mathcal{T}_+ \in (0, +\infty]$ as the maximal backward and forward existence times of the FLRW solution, respectively.

\begin{lemma}[Power identity\footnote{We refer to it as the “power identity” because it roughly corresponds to the vanishing of time derivative of the energy.}]
	\label{proposition:3.2++}
	The following identity holds for all $t\in(\mathcal{T}_-,\mathcal{T}_+)$
		\begin{equation*} \mathcal{P}:=H\left(\biggl(2H^2-1-\frac{\dot{H}}{3H^2}\biggr)\dot{\phi}^2 - 8H^3\phi \dot{\phi} \right.  
		\left. - 12H^6\biggl(1+\frac{\dot{H}}{H^2}\biggr)\phi^2\right) = 0. 
	\end{equation*}
\end{lemma}

\begin{proof}
	We begin by expanding \eqref{e:dtctnt}, and then use \eqref{eq:2.6++} to eliminate the second derivative $\ddot{\phi}$. This yields
	\begin{equation*}
		H\bigl((2H^2-1)\dot{\phi}^2 - 8H^3\phi \dot{\phi} - 12H^4\dot{H}\phi^2 - 12H^6\phi^2 + (4H \phi\dot{\phi} - 1) \dot{H}\bigr) = 0. 
	\end{equation*}
Finally, substituting the expression from \eqref{eq:2.5++} to replace $4H\phi\dot{\phi} - 1$ completes the derivation. 
	\end{proof}
 
\begin{lemma}
\label{lemma:3.2++++}
     If $H(T)\neq0$ for some constant $T$, then $
        \dot{\phi}(T)\neq0$. 
\end{lemma}
\begin{proof}
    If $H(T)\neq0$, then
    \begin{equation*}
        \sqrt{(6\phi(T) H^3(T))^2+3H^2(T)}>6|\phi(T) H^3(T)|. 
    \end{equation*}
   Substituting this into \eqref{e:dotphi}, when $\iota=0$, we obtain
    \begin{equation*}
    \begin{split}
        \dot{\phi}(T)&=-6\phi(T) H^3(T)+\sqrt{(6\phi(T) H^3(T))^2+3H^2(T)}\\
        &>6|\phi(T) H^3(T)|-6\phi(T) H^3(T)\geq0,
         \end{split}
    \end{equation*}
   	when $\iota=1$, similarly,
       \begin{equation*}
         \begin{split}
        \dot{\phi}(T)&=-6\phi(T) H^3(T)-\sqrt{(6\phi(T) H^3(T))^2+3H^2(T)}\\
       & <-6\phi(T) H^3(T)-6|\phi(T) H^3(T)|\leq0.
            \end{split}
    \end{equation*}
   		In either case, $\dot{\phi}(T) \neq 0$. 
\end{proof}

\subsubsection{Lower bound of $H$ for $t \in (0,\mathcal{T}_+)$}
We begin by establishing a lower bound for $H$ on the interval $t \in (0, \mathcal{T}_+)$. Let $B_1(t)$ be defined as \begin{equation}
  B_1:=\dot{H}+5H^2.
  \label{B_1}
\end{equation}

\begin{lemma}\label{lemma:3.2}
 If there exists a constant $T_0 > 0$ such that $B_1(t) > 0$ for all $t \in (0, T_0)$, then
	\begin{equation*}
		H(t) > \frac{1}{5t + \frac{1}{\beta}}, \quad \text{for all } t \in (0,T_0).
	\end{equation*}
\end{lemma}
\begin{proof}
	The assumption implies that $H$ satisfies, for all $t \in (0,T_0)$ 
	\begin{align*}
		\dot{H}>-5H^2  \AND
		H|_{t=0}=\beta.
	\end{align*}
	 Consider a comparison function $\underline{H}_+$ defined by   \textit{Riccati-type ODE}, 
	\begin{align}
		\dot{\underline{H}}_+  = -5\underline{H}_+^2 \AND
		\underline{H}_+|_{t=0} = \beta. \label{e:cpsl2}
	\end{align}  
 Theorem \ref{theorem:C} guarantees that $H(t) > \underline{H}_+(t)$ for all $t \in (0, T_0)$. Solving the Riccati-type ODE  \eqref{e:cpsl2} yields
	\begin{equation}
		H(t) > \underline{H}_+(t) = \frac{1}{5t + \frac{1}{\beta}}, \quad \text{for all } t \in (0,T_0), 
        \label{eq:underlineH+}
	\end{equation}
	which completes the proof.
\end{proof}

\begin{lemma}\label{lemma:3.3}
If $\beta>0$ and $B_1(0) > 0$, then $B_1(t) > 0$ for all $t \in (0, \mathcal{T}_+)$.
\end{lemma}
	\begin{proof}
		Since $B_1(0) > 0$ and $B_1(t)$ is continuous (by the definition \eqref{B_1} and Proposition \ref{t:locext}), there exists a constant $T\in(0, \mathcal{T}_+]$, such that $B_1(t)>0$ for all $t\in(0, T)$. Define
			\begin{equation}\label{e:Tmax}
			T_\text{max}:=\sup\{T\in(0,\mathcal{T}_+] \;|\;B_1(t)>0 \;\;\text{for all } t \in (0, T)\}.
			\end{equation}

		To prove the lemma, it suffices to show that $T_{\text{max}} = \mathcal{T}_+$. Suppose, for contradiction, that $T_{\text{max}} < \mathcal{T}_+$. Then because of the continuity of $B_1$, we must have $B_1(T_{\text{max}}) = 0$; otherwise, this would contradict the definition of $T_{\text{max}}$ in \eqref{e:Tmax}.

		By Lemma \ref{lemma:3.2}, the positivity of $B_1$ on $(0,T_{\text{max}})$ implies
		\begin{equation*}
			H(t) > \frac{1}{5t + \frac{1}{\beta}}, \quad \text{for all } t \in (0, T_{\text{max}}).
		\end{equation*} 
	Taking the limit as $t \to T^-_{\text{max}}$ and using continuity of $H$, we obtain
		\begin{equation*}
			H(T_{\text{max}}) \geq \frac{1}{5T_{\text{max}} + \frac{1}{\beta}} > 0,
		\end{equation*}
		where the last inequality follows from the fact that $T_{\text{max}} <\mathcal{T}_+\in(0,+\infty]$.

		Next, evaluate the power identity at $t = T_{\text{max}}$. By Lemma \ref{proposition:3.2++}, we have $\mathcal{P}(T_{\text{max}}) = 0$, and  taking $B_1(T_{\text{max}}) = 0$  into accounts yields 
		\begin{align}
			0=&\mathcal{P}(T_{\text{max}} )=\left. H \Biggl(\biggl(2H^2+\frac{2}{3}\biggr)\dot{\phi}^2 - 8H^3 \phi \dot{\phi}  + 48H^6 \phi^2 \Biggr) \right|_{t=T_{\text{max}} } \notag \\
			=&\left. H \Biggl( 2H^2 \dot{\phi} ^2 + 2\biggl(\frac{\dot{\phi} }{\sqrt{3}} - 2\sqrt{3}H^3 \phi \biggr)^2   
			+ 24H^6 \phi^2  \Biggl) \right|_{t=T_{\text{max}} } >0 . \label{e:0>0}
		\end{align}  
		Let us justify \eqref{e:0>0}. By Lemma \ref{lemma:3.2++++}, we have $\dot{\phi}(T_{\text{max}}) \neq 0$, so $H^2 \dot{\phi}^2 |_{t = T_{\text{max}}} > 0$  is strictly positive. Since all the remaining terms are manifestly non-negative and $H(T_{\text{max}}) > 0$, the entire expression is strictly positive, leading to the contradiction \eqref{e:0>0}. 
		Hence, our assumption that $T_{\text{max}} < \mathcal{T}_+$ must be false, and we conclude that $T_{\text{max}} = \mathcal{T}_+$. This completes the proof. 
	\end{proof}

\begin{lemma} \label{t:dataB}
	Under the initial conditions \eqref{eq:2.12}--\eqref{eq:2.11}, or alternatively \eqref{eq:1.13++++}--\eqref{1.14++++}, we have $B_1(0) > 0$.
\end{lemma}

\begin{proof}
	We first verify that the initial conditions \eqref{eq:2.12}--\eqref{eq:2.11} imply $B_1(0) > 0$. From \eqref{eq:2.7} and \eqref{eq-dotH}, together with the initial data,  and noting that $\beta > 0$, we compute
	\begin{equation*}
		\dot{\phi}(0) = \sqrt{3}\beta,
	\end{equation*}
	and
	\begin{equation}
		\dot{H}(0) = 6\beta^4 - 3\beta^2.
		\label{eq:2.17++}
	\end{equation}
	Substituting these into the definition of $B_1$ in \eqref{B_1}, we find
	\begin{equation*}\label{e:B0.a}
		B_1(0) = \dot{H}(0) + 5H^2(0)  = 6\beta^4 + 2\beta^2 > 0.
	\end{equation*}
	
	Next, we verify that the alternative initial conditions \eqref{eq:1.13++++}--\eqref{1.14++++} also yield $B_1(0) > 0$. From \eqref{1.14++++} and \eqref{eq:2.7}, we have $\dot{\phi}(0) = \gamma > 0$. In addition, \eqref{eq-dotH} implies $\dot{H}(0) = \kappa$. Therefore, \eqref{eq:1.13++++} implies
	\begin{equation}\label{e:B0.b}
		B_1(0) = \dot{H}(0) + 5H^2(0) = \kappa + 5\beta^2 > 0.
	\end{equation}
	We conclude this lemma.
\end{proof}

\begin{proposition}
	\label{proposition:3.4}
	Under the initial conditions \eqref{eq:2.12}--\eqref{eq:2.11}, or alternatively \eqref{eq:1.13++++}--\eqref{1.14++++},  we have
	\begin{equation*}
		H(t) > \frac{1}{5t + \frac{1}{\beta}}, \quad \text{for all } t \in (0, \mathcal{T}_+).
	\end{equation*}
\end{proposition}

\begin{proof}
	 Lemma \ref{t:dataB} implies $B_1(0)>0$, the result then follows directly from Lemmas \ref{lemma:3.2} and \ref{lemma:3.3}.
\end{proof}

	\subsubsection{Lower Bound of $H$ for $t \in (\mathcal{T}_-, 0)$}\label{s:Hlrbd1}
	We now derive a lower bound for $H$ on the interval $t \in (\mathcal{T}_-, 0)$.  Define $B_2(t)$ as
    \begin{equation}
          B_2:=\dot{H}-6H^4-6H^2.
          \label{B_2}
    \end{equation}
	
	\begin{lemma}
		\label{lemma:3.6++}
		If there exists a constant $T_0 < 0$ such that $B_2(t) < 0$ for all $t \in (T_0, 0)$, then
		\begin{equation*}
			H(t) > \frac{1}{\mathtt{S}^{-1}\left(\mathtt{S}(1/\beta) - 6t\right)} > \frac{1}{\frac{1}{\beta}  - 6t}, \quad \text{for all } t \in (T_0,0),
		\end{equation*}
		where the transcendental function
		$ \mathtt{S}(x) = x + \arctan x$
		is strictly increasing and therefore invertible.
	\end{lemma}
	
	\begin{proof}
		The proof is similar to that of Lemma \ref{lemma:3.2}. Define a comparison function $\underline{\mathcal{H}}_-(t)$ satisfying the initial value problem
		\begin{equation}
			\dot{\underline{\mathcal{H}}}_- = 6\underline{\mathcal{H}}_-^4 + 6\underline{\mathcal{H}}_-^2, \quad \underline{\mathcal{H}}_-(0) = \beta.
			\label{eq:3.25++}
		\end{equation}
		The comparison theorem (Theorem \ref{theorem:C}) ensures that $H(t)>\underline{\mathcal{H}}_-(t)$. 
		
		We now solve the ODE \eqref{eq:3.25++}. Rewriting it yields
		\begin{equation*}
		           \left(\frac{1}{\underline{\mathcal{H}}_-^2} - \frac{1}{\underline{\mathcal{H}}_-^2 + 1}\right) \dot{\underline{\mathcal{H}}}_- = 6. 
		\end{equation*}
        Integrating this equation with the initial data, we get
		\begin{equation}
			\frac{1}{\underline{\mathcal{H}}_-(t)} + \arctan{\frac{1}{\underline{\mathcal{H}}_-(t)}} = -6t + \frac{1}{\beta} + \arctan{\frac{1}{\beta}}.
			\label{eq:3.28++}
		\end{equation}
		Hence,
        \begin{equation*}
            	\underline{\mathcal{H}}_-(t) = \frac{1}{\mathtt{S}^{-1}\left(\mathtt{S}(1/\beta) - 6t\right)}.
        \end{equation*}
	
		To estimate a lower bound for $\underline{\mathcal{H}}_-(t)$, observe that since $\mathtt{S}$ is strictly increasing for all $t \in (T_0, 0)$, so is its inverse $\mathtt{S}^{-1}$ and $\underline{\mathcal{H}}_-(t)$. Furthermore, we have
        \begin{equation*}
            		\arctan{\frac{1}{\underline{\mathcal{H}}_-(t)}} > \arctan{\frac{1}{\underline{\mathcal{H}}_-(0)}} = \arctan{\frac{1}{\beta}}.
        \end{equation*}
		Substituting this into \eqref{eq:3.28++} yields
		\begin{equation*}
			H(t) > \underline{\mathcal{H}}_-(t) > \frac{1}{\frac{1}{\beta} - 6t}, \quad \text{for all } t \in (T_0, 0).
		\end{equation*}
        This completes the proof. 
	\end{proof}

\begin{lemma}
	\label{lemma:3.7++}
   
	If $B_2(0)<0$ and $\beta>0$, we have $B_2(t) < 0,  \text{for all } t \in (\mathcal{T}_-,0).$
	
\end{lemma}
	 
\begin{proof}
	The proof is similar to that of Lemma \ref{lemma:3.3}. 
    Since $ B_2(0) < 0 $ and $ B_2(t) $ is continuous (by Definition \eqref{B_2} and Proposition \ref{t:locext}), there exists a constant $ T \in [\mathcal{T}_-, 0) $ such that $ B_2(t) < 0 $ for all $ t \in (T, 0) $. Define  
    \begin{equation}\label{e:Tmin+}
        T_{\text{min}} := \inf\{T \in [\mathcal{T}_-, 0) \mid B_2(t) < 0 \text{ for all } t \in (T, 0)\}.
    \end{equation}

 As in the proof of Lemma \ref{lemma:3.3}, we prove $ T_{\text{min}} = \mathcal{T}_- $ by contradiction. Suppose $ T_{\text{min}} > \mathcal{T}_- $. By the continuity of $ B_2 $, we must have $ B_2(T_{\text{min}}) = 0 $; otherwise, this would contradict the definition of $ T_{\text{min}} $ in \eqref{e:Tmin+}.

By Lemma \ref{lemma:3.6++}, the negativity of $B_2$ on  $(T_{\text{min}}, 0)$ implies  
\begin{equation*}
	H(t) > \frac{1}{\frac{1}{\beta} - 6t}, \quad \text{for all } t \in (T_{\text{min}}, 0).
\end{equation*}
Taking the limit as $t \to T_{\text{min}}^+$ and using the continuity of $H$, we obtain
\begin{equation*}
	H(T_{\text{min}}) \geq \frac{1}{\frac{1}{\beta} - 6T_{\text{min}}} > 0,
\end{equation*}
where the denominator is strictly positive, since $T_{\text{min}} < 0$ and cannot reach infinity due to the fact that $T_{\text{min}} > \mathcal{T}_{-} \in [-\infty, 0)$.

    Evaluating Lemma \ref{proposition:3.2++} at $ t = T_{\text{min}} $ and substituting $ B_2(T_{\text{min}}) = 0 $, we have  
    \begin{align}
        0 &= \mathcal{P}(T_{\text{min}})  =  H \bigl( -3\dot{\phi}^2 - 8H^3 \phi \dot{\phi} - 84H^6 \phi^2 - 72H^8 \phi^2 \bigr) \big|_{t=T_{\text{min}}} \notag \\
        &=   H \bigl( -2\dot{\phi}^2 - (\dot{\phi} + 4H^3 \phi)^2 - 68H^6 \phi^2 - 72H^8 \phi^2 \bigr) \big|_{t=T_{\text{min}}}. \label{eq:contradiction}
    \end{align}  

    We now analyze the sign of \eqref{eq:contradiction}. Lemma \ref{lemma:3.2++++} implies $ \dot{\phi}(T_{\text{min}}) \neq 0 $, so $ -2\dot{\phi}^2(T_{\text{min}}) < 0 $.  
        Meanwhile, all the other terms are non-positive, and $ H(T_{\text{min}}) > 0 $.   
    Thus, the right-hand side of \eqref{eq:contradiction} is strictly negative, leading to the contradiction $ 0 < 0 $.   
    Therefore, $ T_{\text{min}} = \mathcal{T}_- $, completing the proof.  
\end{proof}

\begin{proposition}
	\label{proposition:3.8++}
	Under the initial conditions \eqref{eq:2.12}--\eqref{eq:2.11}, we have
	\begin{equation*}
		H(t) > \frac{1}{\frac{1}{\beta} - 6t}, \quad \text{for all } t \in (\mathcal{T}_-, 0).
	\end{equation*} 
\end{proposition}

\begin{proof}
	As shown in Lemma \ref{t:dataB}, the initial conditions \eqref{eq:2.12}--\eqref{eq:2.11} imply \eqref{eq:2.17++}. Substituting \eqref{eq:2.17++} into the definition of $B_2$ in \eqref{B_2}, and noting that $\beta > 0$, we obtain $B_2(0) = -9\beta^2 < 0$. 
	The result then follows directly from Lemmas \ref{lemma:3.6++} and \ref{lemma:3.7++}.
\end{proof}

\subsubsection{Upper bound of $H$ for $t \in (0,\mathcal{T}_+)$}

We now establish an upper bound for $H$ on the interval $(0, \mathcal{T}_+)$.  Define $B_3(t)$ as
\begin{equation}
      B_3:=\dot{H}+H^2.
      \label{B_3}
\end{equation}

\begin{lemma}
	\label{lemma:3.1}
	If $\dot{\phi}(0) > 0$,  the solution $\phi$ satisfies
	\begin{equation} 
		\dot{\phi}(t) = -6 \phi(t) H^3(t) + \sqrt{\left(6\phi(t) H^3(t)\right)^2 + 3 H^2(t)}>0
		\label{eq:3.4}
	\end{equation}
	for all $t \in (\mathcal{T}_-,\mathcal{T}_+)$. 
\end{lemma} 
\begin{proof} 
	Propositions \ref{proposition:3.4} and \ref{proposition:3.8++} ensure  $H(t) > 0$ for all $t \in (\mathcal{T}_-, \mathcal{T}_+)$. From the proof of Lemma \ref{lemma:3.2++++} and  \eqref{e:dotphi},  we obtain for $t \in (\mathcal{T}_-, \mathcal{T}_+)$, 
	\begin{equation}\label{e:dphisgn}
		\dot{\phi}(t)
		\begin{cases}
			>0  ,\quad \text{for } \iota=0, \\
			<0 , \quad \text{for } \iota=1 . 
		\end{cases}
	\end{equation} 
	Since $\dot{\phi} \in C^1((\mathcal{T}_-, \mathcal{T}_+))$, the local existence theorem (Theorem \ref{t:locext}) and $\dot{\phi}(0) > 0$ guarantee that there exists a constant $T > 0$ such that $\dot{\phi}(t) > 0$ for $t \in (-T, T)$. 
	Therefore, $\dot{\phi}$ must be given by the $\iota = 0$ branch of \eqref{e:dotphi} throughout $(\mathcal{T}_-, \mathcal{T}_+)$; otherwise, a sign change would contradict the continuity of $\dot{\phi}$ implied by \eqref{e:dphisgn}. This completes the proof. 
\end{proof}

\begin{lemma}\label{lem-sign-phi}
	Under the initial conditions \eqref{eq:2.12} and \eqref{eq:2.11}, we have
	\begin{equation*}
		\phi(t) > 0, \quad \text{for all } t \in (0,\mathcal{T}_+).
	\end{equation*}
	\begin{equation*}
		\phi(t) < 0, \quad \text{for all } t \in (\mathcal{T}_-,0).
	\end{equation*}
\end{lemma}

\begin{proof}
	
	Combining Lemma \ref{lemma:3.1} with \eqref{eq:2.11}, we obtain $\dot{\phi}(t) > 0$ for all $t \in (\mathcal{T}_-, \mathcal{T}_+)$, so $\phi$ is strictly increasing. Together with the initial condition \eqref{eq:2.12}, which implies $\phi(0) = 0$, it follows that $\phi(t) > 0$ for $t \in (0, \mathcal{T}_+)$ and $\phi(t) < 0$ for $t \in (\mathcal{T}_-, 0)$. 
\end{proof}

We are now in a position to establish an upper bound for $H(t)$ on the interval $(0, \mathcal{T}_+)$.

\begin{lemma}\label{lemma:3.5a}
    If there exists $ T_0 >0 $ such that $B_3(t) < 0$ for all $ t \in (0, T_0) $, then
    \begin{equation*}
        H(t) < \frac{1}{t + \frac{1}{\beta}}, \quad \text{for all }  t \in (0, T_0).
    \end{equation*}
\end{lemma}
\begin{proof}
	The proof is similar to that of Lemma \ref{lemma:3.2}. Consider the comparison function $\bar{H}_+(t)$ defined as the solution to the Riccati-type ODE
    \begin{equation}
        \dot {\bar H}_+=-\bar H_+^2 \quad \text{and} \quad \bar H_+|_{t=0}=\beta \implies
         \bar{H}_+(t) = \frac{1}{t + \frac{1}{\beta}}. 
         \label{eq:barH+}
    \end{equation} 
    Then, from Theorem \ref{theorem:C}, we deduce $H(t) < \bar{H}_+(t)$ on the interval $(0, T_0)$, 
    which completes the proof.
\end{proof}

\begin{lemma}\label{lemma:3.5b}

	Under the initial conditions \eqref{eq:2.12}--\eqref{eq:2.11}, if   $B_3(0)<0$, then  $
	B_3(t)< 0$,  for all   $t \in (0, \mathcal{T}_+)$. 
\end{lemma}

\begin{proof}
	The proof is similar to that of Lemma \ref{lemma:3.3}. Since $B_3(0) < 0$ and $B_3(t)$ is continuous (\eqref{B_3} and Proposition \ref{t:locext}), there exists a constant $T \in (0, \mathcal{T}_+]$ such that $
	B_3(t) < 0$ for all $t \in (0, T)$. 
	Define
	\[
	T_{\text{max}} := \sup\left\{T \in (0, \mathcal{T}_+] \;\middle|\; \dot{H}(t) + H^2(t) < 0 \text{ for all } t \in (0, T)\right\}.
	\]
	We claim that $T_{\text{max}} = \mathcal{T}_+$. Suppose, for contradiction, that $T_{\text{max}} < \mathcal{T}_+$. From the continuity of $B_3(t)$, we conclude $B_3(T_{\text{max}}) = 0$. 
	
	Lemma \ref{lemma:3.5a} and the continuity of $H$ imply
	\begin{equation}\label{e:H<bt}
		H(T_{\text{max}}) \leq \frac{1}{T_{\text{max}} + \frac{1}{\beta}} < \beta < \frac{\sqrt{3}}{3}.
	\end{equation}
	Using \eqref{e:H<bt} and Lemma \ref{lemma:3.1}, we find
	\begin{equation}\label{e:Hest1}
		\left.\left(2H^2 - \frac{2}{3}\right) \dot{\phi}^2\right|_{t = T_{\text{max}}} < 0.
	\end{equation}
	Furthermore, Proposition \ref{proposition:3.4}  (due to the initial conditions \eqref{eq:2.12}--\eqref{eq:2.11}), together with Lemmas \ref{lemma:3.1} and \ref{lem-sign-phi}, implies
	\begin{equation}\label{e:Hest2}
		H(T_{\text{max}}) > 0 \quad \text{and} \quad  -8H^3 \phi \dot{\phi}\big|_{t = T_{\text{max}}} < 0.
	\end{equation}
	
	Substituting $\dot{H}(T_{\text{max}}) + H^2(T_{\text{max}}) = 0$ into the power identity in Lemma \ref{proposition:3.2++} evaluated at $t = T_{\text{max}}$, and using \eqref{e:Hest1}--\eqref{e:Hest2}, we obtain
	\[
	0 = \left. H\Biggl(\biggl(2H^2 - \frac{2}{3}\biggl)\dot{\phi}^2 - 8H^3 \phi \dot{\phi} \Biggl) \right|_{t = T_{\text{max}}} < 0,
	\]
	which is a contradiction. Hence, $T_{\text{max}} = \mathcal{T}_+$, completing the proof.
\end{proof}

\begin{lemma}\label{t:dH+H2}
	 	Under the initial conditions \eqref{eq:2.12}--\eqref{eq:2.11}, or alternatively \eqref{eq:1.13++++}--\eqref{1.14++++}, we have $B_3(0)< 0$. 
	 \end{lemma}
	 \begin{proof}
	 	The proof of Lemma \ref{t:dataB} together with \eqref{eq:2.17++} and \eqref{e:B0.b} implies
	 	\begin{align*} 	B_3(0) =		 
	 		\begin{cases}
	  6\beta^4   -2\beta^2\overset{\eqref{eq:2.13}}{<}0, \quad &\text{for data \eqref{eq:2.12}--\eqref{eq:2.11}} , \\
	 	 \kappa +  \beta^2 \overset{\eqref{eq:1.13++++}}{<}0,   &\text{for data \eqref{eq:1.13++++}--\eqref{1.14++++}},
	 		\end{cases}  
	 	\end{align*} 
	which completes the proof.  	 	 
	 \end{proof}

\begin{proposition}\label{proposition:3.5}
    Under the initial conditions \eqref{eq:2.12}--\eqref{eq:2.11}, or alternatively \eqref{eq:1.13++++}--\eqref{1.14++++}, $H(t)$ has the upper bound
    \begin{equation*}
        H(t) < \frac{1}{t + \frac{1}{\beta}}, \quad \text{for all} \quad t \in (0, \mathcal{T}_+).
    \end{equation*}
\end{proposition}
\begin{proof}
  This proposition immediately follows from Lemmas \ref{lemma:3.5a}-\ref{t:dH+H2}.
\end{proof}

\subsubsection{Upper Bound of $H$ for $t \in (\mathcal{T}_-, 0)$}

We now turn to establishing an upper bound for $H$ on the interval $(\mathcal{T}_-, 0)$. Define $B_4(t)$ as
\begin{equation}
      B_4:=\dot{H}-3H^2+\frac{3}{2}.
      \label{B_4}
\end{equation}

\begin{lemma}\label{lemma:3.10a}
	Suppose there exists a constant $T_0 < 0$ such that $B_4(t)>0$ for all $t \in (T_0, 0)$. Then
	\begin{equation*}
		H(t) < \frac{2 \beta +\sqrt{2}-\left(\sqrt{2}-2 \beta \right) e^{3 \sqrt{2} t}}{\sqrt{2} \left(2 \beta  +\sqrt{2}+\left(\sqrt{2}-2 \beta \right) e^{3 \sqrt{2} t}\right)}<\frac{1}{\sqrt{2}}, \quad \text{for all } t \in (T_0, 0).
	\end{equation*} 
\end{lemma}

\begin{proof}
	Corresponding to $B_4(t)=0$, we define $\bar H_-$ to satisfy 
    \begin{equation*}
        \dot{\bar H}_-(t) = -\frac{3}{2} + 3 \bar H^2_-(t), \quad \bar H_-(0)=\beta.\quad \text{for all } t \in (T_0, 0).
    \end{equation*}
    Directly solving the above equation yields
\begin{equation}
   \bar H_-(t) = \frac{2 \beta +\sqrt{2}-\left(\sqrt{2}-2 \beta \right) e^{3 \sqrt{2} t}}{\sqrt{2} \left(2 \beta  +\sqrt{2}+\left(\sqrt{2}-2 \beta \right) e^{3 \sqrt{2} t}\right)}, \quad \text{for all } t \in (T_0, 0).
   \label{eq:barH--}
\end{equation}
The comparison theorem \ref{theorem:C} shows $\bar H_-$ is an upper bound for $H$.
\end{proof}

\begin{lemma}\label{lemma:3.10b}
	Under the initial conditions \eqref{eq:2.12}--\eqref{eq:2.11}, if   $B_4(0)>0$, then   $B_4(t)>0 $, for all $ t \in (\mathcal{T}_-,0) $.
\end{lemma}

\begin{proof}
	Since $B_4(0) > 0$ and $\dot{H}(t)$ is continuous (by definition \eqref{B_4} and Proposition \ref{t:locext}), there exists a constant $T \in (\mathcal{T}_-, 0]$ such that $B_4(t) > 0$ for all $t \in (T, 0)$. Define
	\begin{equation*}
	T_{\text{min}} := \inf\left\{T \in (\mathcal{T}_-, 0] \;\middle|\; B_4(t) > 0 \text{ for all } t \in (T, 0)\right\}.
	\end{equation*}
	We claim that $T_{\text{min}} = \mathcal{T}_-$. Suppose, for contradiction, that $T_{\text{min}} > \mathcal{T}_-$. By continuity, we then have $
	B_4({T_{\text{min}}})=(\dot{H}-3H^2 + \frac{3}{2})|_{T_{\text{min}}}  = 0$. 
	Substituting this into the power identity (Lemma \ref{proposition:3.2++}) evaluated at $t = T_{\text{min}}$ yields
	\begin{equation}\label{eq:key_identity++}
		0 = H \left( \Bigl(2H^2 + \frac{1}{2H^2} - 2 \Bigr)\dot{\phi}^2 - 8H^3 \phi \dot{\phi} - 48H^6 \phi^2 + 18H^4 \phi^2 \right)\Big|_{t = T_{\text{min}}}.
	\end{equation}
	
	From the initial conditions \eqref{eq:2.12}--\eqref{eq:2.11} and hence Proposition \ref{proposition:3.8++}, we have $H|_{T_{\text{min}}} > 0$. We now analyze each term in \eqref{eq:key_identity++}: 
    \begin{enumerate}[leftmargin=*]
	\item\label{c:1} \textbf{First term}: By the Cauchy-Schwarz inequality,
	\begin{equation*}
		\biggl(2H^2 + \frac{1}{2H^2 }\biggr) \bigg|_{t=T_{\text{min}}}
		\geq 2 \overset{\text{Lem.}\ref{lemma:3.1}}{\implies}    \biggl(2H^2 + \frac{1}{2H^2 }-2\biggr)\dot{\phi}^2 \bigg|_{t=T_{\text{min}}}\geq 0. 
	\end{equation*} 
	
	\item\label{c:2} \textbf{Second and third terms}: From \eqref{eq:3.4},  Lemma \ref{lem-sign-phi} (i.e., $\phi|_{t=T_{\text{min}}} < 0$) and $H|_{t=T_{\text{min}}} > 0$, we obtain $
	\dot{\phi} |_{t=T_{\text{min}}}> -12\phi H^3 |_{t=T_{\text{min}}}$. 
	Therefore,
	\begin{align*}
		(-8H^3 \phi \dot{\phi} - 48 H^6 \phi^2 )|_{t=T_{\text{min}}} >& (-8H^3 \phi (-12\phi H^3 ) - 48 H^6 \phi^2 )|_{t=T_{\text{min}}}  \notag  \\
		=& (48 H^6 \phi^2 )|_{t=T_{\text{min}}}> 0.
	\end{align*}
	
	\item\label{c:3} \textbf{Fourth term}: Since $H^4({T_{\text{min}}})> 0$ and $\phi^2({T_{\text{min}}}) > 0$, we obtain $
	18H^4 \phi^2|_{t=T_{\text{min}}}> 0$. 
\end{enumerate}

	Combining the above estimates \eqref{c:1}--\eqref{c:3}, all the terms in \eqref{eq:key_identity++} are strictly positive, which leads to the contradiction $0 > 0$. Hence,  we conclude that $T_{\text{min}} = \mathcal{T}_-$. Therefore, $B_4(t) > 0$ for all $t \in (\mathcal{T}_-, 0)$.
\end{proof}

\begin{proposition}\label{proposition:3.10}
Under the initial conditions \eqref{eq:2.12}--\eqref{eq:2.11}, $ H(t) $ has the upper bound:
\begin{equation*}
    H(t) < \frac{2 \beta +\sqrt{2}-\left(\sqrt{2}-2 \beta \right) e^{3 \sqrt{2} t}}{\sqrt{2} \left(2 \beta  +\sqrt{2}+\left(\sqrt{2}-2 \beta \right) e^{3 \sqrt{2} t}\right)} <\frac{1}{\sqrt{2}} <1,  \quad \text{for all} \quad t \in (\mathcal{T}_-, 0).
\end{equation*}
\end{proposition}
\begin{proof}
Using \eqref{eq:2.17++}, we obtain $
		B_4(0) = 6\beta^4 - 3\beta^2 +\frac{3}{2} -3\beta^2 >  \frac{1}{6}>0$.  
	This proposition immediately follows from Lemmas \ref{lemma:3.10a} and \ref{lemma:3.10b}.
\end{proof}

\subsection{Improved lower bound for $H$ on $t \in (\mathcal{T}_-, 0)$}

In \S\ref{s:Hlrbd1}, we establish a lower bound for $H$ on the interval $(\mathcal{T}_-, 0)$. However, that estimate is insufficient to fully capture the essential behavior of $H$.  
In this section, we derive an improved \textit{a posteriori} estimate for the lower bound of $H$, which provides a more accurate description of its profile. We first state the main result in Proposition~\ref{t:Hlrbd2}, and defer its proof to the end of the section. To proceed, we define $B_5$ as 
	\begin{align}\label{e:B5.a}
		B_5(t):=\begin{cases}
			\dot{H}(t)-10 H^2(t)+ \frac{454 \beta}{45} H(t) , \quad &\text{if }0< \beta \leq \sqrt{\frac{5}{27}} , \\
			\dot{H}(t)-10 H^2(t)+4H(t),  &\text{if }\sqrt{\frac{5}{27}}< \beta <\frac{\sqrt{3}}{3}  .
		\end{cases} 
	\end{align}
\begin{proposition}\label{t:Hlrbd2}
	Under the initial conditions \eqref{eq:2.12}--\eqref{eq:2.11}, $H(t)$ satisfies
	\begin{equation*}
		H(t) >\mathfrak{L}(t):= \begin{cases}
			\displaystyle	\frac{227 \beta }{2 e^{\frac{454 \beta  t}{45}}+225} >\beta , \quad &\text{if }0< \beta \leq \sqrt{\frac{5}{27}} , \\
			\displaystyle	\frac{2 \beta }{5 \beta - (5 \beta -2) e^{4 t} } >\frac{2}{5} , \quad &\text{if }\sqrt{\frac{5}{27}}< \beta <\frac{\sqrt{3}}{3}  ,
		\end{cases}
		\quad \text{for all } t \in (\mathcal{T}_-, 0).
	\end{equation*} 
\end{proposition}	
	
\begin{lemma}\label{t:Hphidphi}
Under the initial conditions \eqref{eq:2.12}--\eqref{eq:2.11},	if $\dot{\phi}$ satisfies \eqref{eq:3.4}, then for all    $t\in (\mathcal{T}_-,0)$, 
 \begin{equation*}
 H\phi\dot{\phi} >-12\phi^2H^4 + \sqrt{3} \phi H^2.
 \end{equation*} 
\end{lemma}
\begin{proof} 
 	From   Proposition~\ref{proposition:3.8++}, we know that $H(t) > 0$ for all $t \in (\mathcal{T}_-, 0)$. Moreover, Lemma~\ref{lem-sign-phi} implies $H\phi = -\sqrt{H^2 \phi^2}$ on this interval. Through \eqref{eq:3.4}, we obtain
 \begin{equation*}
H\phi \dot{\phi} 
=   -6 \phi^2 H^4  -   \sqrt{\left(6\phi^2  H^4 \right)^2 + 3 H^4 \phi^2 }  >-12\phi^2H^4 + \sqrt{3} \phi H^2 ,
 \end{equation*}
which completes the proof.
\end{proof}

\begin{lemma}\label{t:PI2}
	Let
\begin{equation*}
	B_5(t):=	\dot{H}(t)-\mu H^2(t)+bH(t),  \quad (\mu>0,b>0), 
\end{equation*} 
If $B_5(T)=0$ and $H(T)\geq k>0$ for some $T \in (\mathcal{T}_-,0)$, then the power identity (as given in Lemma \ref{proposition:3.2++})
\begin{align*}
	&0=\mathcal{P}|_{t=T}< \notag  \\
	&\left.H\left(A\dot{\phi}^2  -\left(\sqrt{12\left(1+\mu-c-\frac{b	}{k}\right) } \phi H^3 +qH \right)^2 +\left( q^2   -\frac{ 8-c }{4} \right)  H^2  \right)  \right|_{t=T} 
\end{align*}
where $c>0$ is some constant,   
\begin{align*}
	A:= 2H^2-1-\frac{\mu }{3 }+\frac{ b}{3H}+ \frac{ 8-c }{12} \AND
	q^2:= \frac{c^2}{16\bigl(1+\mu-c-\frac{b}{k}\bigr)}  ,
\end{align*}
 and $\mu, b, c, k$ satisfy
 \begin{equation*}
     1+\mu-c-\frac{b}{k}>0.
 \end{equation*}
\end{lemma}

\begin{proof}
The proof is based on a direct computation. Using  Lemmas \ref{proposition:3.2++} and \ref{t:Hphidphi}, we obtain 
 \begin{align*}
 	\mathcal{P}(T)  
 	=& H\left(\biggl(2H^2-1-\frac{\dot{H}}{3H^2}\biggr)\dot{\phi}^2 -( 8-c) H^3\phi \dot{\phi}   -c H^3\phi \dot{\phi}   \right.\notag  \\
 	&	\left.\left.
 	- 12H^6 \phi^2	- 12H^4 \dot{H} \phi^2\right)  \right|_{t=T  }\notag \\
 	\overset{\text{Lem. \ref{t:Hphidphi}}}{<}& H\left(\biggl(2H^2-1-\frac{\dot{H}}{3H^2}\biggr)\dot{\phi}^2 -( 8-c) H^3\phi \dot{\phi}   +12c\phi^2H^6 - c\sqrt{3} \phi H^4  \right.\notag  \\
 &	\left.\left.
 	- 12H^6 \phi^2	- 12H^4 \dot{H} \phi^2\right)  \right|_{t=T }\notag \\
 		\overset{\eqref{eq:2.5++}}{=}& H\left(\biggl(2H^2-1-\frac{\dot{H}}{3H^2}+ \frac{ 8-c }{12}\biggr)\dot{\phi}^2  -\frac{ 8-c }{4} H^2  +12(c-1)\phi^2H^6 \right.  \notag  \\
 		& \left.\left.- c\sqrt{3} \phi H^4 
 	- 12H^4 \dot{H} \phi^2\right)  \right|_{t=T }. 
 \end{align*} 
 Substituting  $\dot{H}|_{t=T}= (\mu H^2 - bH)|_{t=T}$ from  $B_5|_{t=T}=0$, we obtain 
 \begin{align*}
  	\mathcal{P}(T)  
 <& H\left(\biggl(2H^2-1-\frac{\mu }{3 }+\frac{ b}{3H}+ \frac{ 8-c }{12}\biggr)\dot{\phi}^2  +12\Bigl(c-1-\mu+\frac{b}{H}\Bigr)\phi^2H^6 \right.  \notag  \\
 & \left.\left. - c\sqrt{3} \phi H^4  	   -\frac{ 8-c }{4} H^2  \right)  \right|_{t=T  }     \notag \\
 = &H\left(\biggl(2H^2-1-\frac{\mu }{3 }+\frac{ b}{3H}+ \frac{ 8-c }{12}\biggr)\dot{\phi}^2  +12\biggl(c-1-\mu+\frac{b}{k}\biggr)\phi^2H^6 \right.  \notag  \\
 & \left.\left. - c\sqrt{3} \phi H^4  	- q^2 H^2 +q^2 H^2   -\frac{ 8-c }{4} H^2  \right)  \right|_{t=T  }        \notag \\
& \hspace{-1cm} =\left.H\left(A\dot{\phi}^2  -\left(\sqrt{12\biggl(1+\mu-c-\frac{b	}{k}\biggr) } \phi H^3 +qH \right)^2 +\left( q^2   -\frac{ 8-c }{4} \right)  H^2  \right)  \right|_{t=T }  , 
 \end{align*}
which completes the proof.
\end{proof}

Now we turn to two specific forms of $B_5$. 
\begin{lemma}\label{t:HlrbdB5.a}
	
	If there exists a constant $T_0 < 0$ such that  $B_5(t)<0$
	for all $t \in (T_0, 0)$, then
	\begin{equation*}
	H(t) > \mathfrak{L}(t)
	\quad \text{for all } t \in (T_0, 0),
\end{equation*} 
where $\mathfrak{L}(t)$ is defined in Proposition \ref{t:Hlrbd2}. 
\end{lemma}    

 	\begin{proof}
 	The proof is similar to that of Lemma \ref{lemma:3.2}. Define a comparison function $\underline{H}_-(t)$ satisfying the initial value problem
 	\begin{equation}
 		\dot{\underline{H}}_- = \begin{cases}
 			 10 \underline{H}_-^2(t)- \frac{454 \beta}{45} \underline{H}_- , \quad &\text{if }0< \beta \leq \sqrt{\frac{5}{27}} , \\
 		 10 \underline{H}_-^2(t)-4\underline{H}_-,  &\text{if }\sqrt{\frac{5}{27}}< \beta <\frac{\sqrt{3}}{3}  ,  
 		\end{cases}  \quad \text{with}\quad  \underline{H}_-(0) = \beta.
 		\label{e:B5.b+}
 	\end{equation}
 Solving \eqref{e:B5.b+} yields $
    \underline{H}_-=\mathfrak{L}(t)$.  
 	Theorem \ref{theorem:C} ensures that 
 	$H(t)>\underline{H}_-(t)$,  thus completes the proof. 
 \end{proof}

 \begin{lemma}
	\label{t:HlbdB5.b}
	If $B_5(0)<0$ and $\beta\in(0, \sqrt{3}/3)$, we have	$B_5(t) < 0$,  for all $t \in (\mathcal{T}_-,0).$
	
\end{lemma}

\begin{proof}
	The proof is similar to that of Lemma \ref{lemma:3.3}. 
	Since $ B_5(0) < 0 $ and $ B_5(t) $ is continuous (by  Proposition \ref{t:locext}), there exists a constant $ T \in [\mathcal{T}_-, 0) $ such that $ B_5(t) < 0 $ for all $ t \in (T, 0) $. Define  
	\begin{equation*}
		T_{\text{min}} := \inf\{T \in [\mathcal{T}_-, 0) \mid B_5(t) < 0 \text{ for all } t \in (T, 0)\}.
	\end{equation*}
	
	As the proof of Lemma \ref{lemma:3.3}, we prove $ T_{\text{min}} = \mathcal{T}_- $ by contradiction. Suppose $ T_{\text{min}} > \mathcal{T}_- $. Since $B_5$ is continuous, we must have $B_5(T_{\text{min}}) = 0$.

 Lemma \ref{t:HlrbdB5.a} implies $B_5<0$, the negativity of $B_5$ on  $(T_{\text{min}}, 0)$ implies  
	\begin{equation*}
	H(t) > \mathfrak{L}(t)
	\quad \text{for all } t \in (T_\text{min}, 0). 
\end{equation*}

	Taking the limit as $t \to T_{\text{min}}^+$ and using the continuity of $H$, we obtain
	\begin{equation*}
		H(T_{\text{min}}) \geq k:= \begin{cases}
		 \beta , \quad &\text{if }0< \beta \leq \sqrt{\frac{5}{27}} , \\
		 \frac{2}{5} , \quad &\text{if } \sqrt{\frac{5}{27}}< \beta <\frac{\sqrt{3}}{3}  .
		\end{cases} 
	\end{equation*}

	For the specific $B_5$ defined in \eqref{e:B5.a} and taking  $\mu, b, c$ as follows, we obtain $q$ by Lemma \ref{t:PI2}, 
	\begin{align*}
		\begin{cases}
			\mu= 10,\quad b= \frac{454 \beta}{45},\quad c=\frac{4}{5} ,   \implies q^2= \frac{9}{25} &\text{for }0< \beta \leq \sqrt{\frac{5}{27}},  \\
			\mu= 10,\quad b= 4,\quad c= \frac{4}{9}   \implies q^2= \frac{1}{45}  &\text{for }\sqrt{\frac{5}{27}}< \beta <\frac{\sqrt{3}}{3}   .
		\end{cases} 
	\end{align*}

In view of $ B_5(T_{\text{min}}) = 0 $, and evaluating Lemma \ref{t:PI2} at $ t = T_{\text{min}} $ where $B_5$ is defined by \eqref{e:B5.a}, we have  
\begin{align*}
	&0=\mathcal{P}|_{t=T_\text{min}}< \notag  \\
&{\footnotesize	 \begin{cases}  
	\left.H\left(\biggl(2 H^2+\frac{454 \beta }{135 H}-\frac{56}{15}\biggr)\dot{\phi}^2  -\left(\sqrt{\frac{4}{3} } \phi H^3 +\frac{3}{5} H \right)^2  -\frac{36}{25}  H^2  \right)  \right|_{t=T_{\text{min}} }     , \quad &\text{if } 0< \beta \leq \sqrt{\frac{5}{27}}  , \\
	\left.H\left(\biggl(2 H^2+\frac{4}{3 H}-\frac{100}{27}\biggr)\dot{\phi}^2  -\left(\sqrt{\frac{20}{3} } \phi H^3 +\frac{1}{3 \sqrt{5}}  H \right)^2   -\frac{28}{15}    H^2  \right)  \right|_{t=T_{\text{min}} }, &\text{if }\sqrt{\frac{5}{27}}< \beta <\frac{\sqrt{3}}{3}  .
	\end{cases}  }
\end{align*}

Lemmas \ref{t:pre1} and \ref{t:pre2} imply $\mathcal{P}|_{t=T_\text{min}}<0$, leading to the contradiction $ 0 < 0 $.   
	Therefore, $ T_{\text{min}} = \mathcal{T}_- $, completing the proof.  
\end{proof}

\begin{lemma}\label{t:dH+H3}
	Under the initial conditions \eqref{eq:2.12}--\eqref{eq:2.11}, we have $B_5(0)< 0$. 
\end{lemma}

\begin{proof} 
		Under the initial conditions \eqref{eq:2.12}--\eqref{eq:2.11},  \eqref{eq:2.17++} implies,    
	\begin{align*} 	
		B_5(0) =		 
		\begin{cases}
			 6 \beta ^4-\frac{131 \beta ^2}{45}<0,
			\quad &\text{if }0< \beta \leq \sqrt{\frac{5}{27}} , \\
			6 \beta ^4-13 \beta ^2+4 \beta<0 ,
			 \quad &\text{if }\sqrt{\frac{5}{27}}< \beta <\frac{\sqrt{3}}{3}  . 
		\end{cases}  
	\end{align*}    
	We complete the proof.  	 	 
\end{proof}

We are now in a position to prove  Proposition \ref{t:Hlrbd2}. 
\begin{proof}[Proof of Proposition \ref{t:Hlrbd2}] 
	The result follows directly from Lemmas \ref{t:HlrbdB5.a}--\ref{t:dH+H3}.
\end{proof}

\subsection{Estimates for $\phi$} 

Having established the bounds for $H$ in the previous section, we now turn our attention to estimating $\phi$. Throughout the proof, we will \textit{frequently make use of the positivity of $H(t)$ for $t \in (\mathcal{T}_-, \mathcal{T}_+)$}, which follows from the initial data $\beta > 0$ and Propositions \ref{proposition:3.4} and \ref{proposition:3.8++}, as well as \textit{the sign of $\phi$} as given in Lemma \ref{lem-sign-phi}, without explicitly stating these facts each time.

\subsubsection{Lower Bound for $\phi$ on $t \in (0, \mathcal{T}_+)$}

We begin with establishing a lower bound for $\phi$ on the interval $t \in (0, \mathcal{T}_+)$.

\begin{proposition}
	\label{proposition:4.2}
	Under the initial conditions \eqref{eq:2.12}--\eqref{eq:2.11},  or alternatively \eqref{eq:1.13++++}--\eqref{1.14++++}, $\phi(t)$ satisfies
	\begin{equation*}
		\phi(t) > \frac{-1 + \sqrt{1 + \frac{24\beta^2}{5}\ln{(1+5\beta t)}}}{4\sqrt{3}\beta^2}, \quad \text{for all } t \in (0,\mathcal{T}_+).
	\end{equation*}
\end{proposition}

\begin{proof}
	Using \eqref{eq:3.4}, \eqref{eq:B1}, and the facts that $H > 0$ and $\phi > 0$ on $(0, \mathcal{T}_+)$, we obtain
	\begin{equation}
		\dot{\phi} > \frac{\sqrt{3}H}{1 + 4\sqrt{3}\phi H^2}.
        \label{eq4.2}
	\end{equation}
	Applying the bounds for $H$ from Propositions \ref{proposition:3.4} and \ref{proposition:3.5} (by the initial conditions \eqref{eq:2.12}--\eqref{eq:2.11} or \eqref{eq:1.13++++}--\eqref{1.14++++}), we find
	\begin{equation*}
		\dot{\phi} > \frac{\sqrt{3}}{5t + \frac{1}{\beta}} \cdot \frac{1}{1 + 4\sqrt{3}\phi\beta^2}, \quad \text{with} \quad \phi(0)=0 \quad \text{(by \eqref{eq:2.12})}.
	\end{equation*}
	Define $\underline{\phi}_+$ to be the solution to the ODE
	\begin{equation}
		\dot{\underline{\phi}}_+ = \frac{\sqrt{3}}{5t + \frac{1}{\beta}} \cdot \frac{1}{1 + 4\sqrt{3}\underline{\phi}_+ \beta^2}, \quad \text{with} \quad  \underline{\phi}_+(0) = 0.
		\label{eq:4.11}
	\end{equation}
	Considering Theorem \ref{theorem:C}, it follows that $\phi(t) > \underline{\phi}_+(t)$ for all $t \in (0, \mathcal{T}_+)$.
	
	To solve \eqref{eq:4.11}, we separate variables and integrate to obtain 
	\begin{equation*}
		2\sqrt{3}\beta^2 \underline{\phi}_+^2 + \underline{\phi}_+ = \frac{\sqrt{3}}{5}\ln(1 + 5\beta t).
	\end{equation*}
	Solving this quadratic equation, subject to the initial condition $\underline{\phi}_+(0) = 0$, yields
	\begin{equation*}
		\underline{\phi}_+(t) = \frac{-1 + \sqrt{1 + \frac{24\beta^2}{5}\ln(1 + 5\beta t)}}{4\sqrt{3}\beta^2},
	\end{equation*}
	which completes the proof.
\end{proof}

    \subsubsection{Lower bound of $\phi$ for $t \in (\mathcal{T}_-,0)$}
We now establish a lower bound for $\phi$ on  $t \in (\mathcal{T}_-, 0)$.

	\begin{proposition}
		\label{proposition:3.13++++}
		Under the initial conditions \eqref{eq:2.12}--\eqref{eq:2.11}, $\phi(t)$ satisfies
		\begin{equation*}
			\phi(t) > \frac{\sqrt{3}}{12}(1-e^{-12t}) ,  \quad \text{for all } t \in (\mathcal{T}_-,0).
		\end{equation*}
	\end{proposition}
	
	\begin{proof}
		From  \eqref{eq:3.4}, \eqref{eq:B2} and the facts that $H>0$ and $\phi < 0$ on $(\mathcal{T}_-,0)$, it follows that
		\begin{equation}
			\dot{\phi}  < -12 \phi  H^3  + \sqrt{3}H , \quad \text{for all } t\in (\mathcal{T}_-,0).
              \label{eq3.13}
		\end{equation} 
	Using the constant upper bound for $H$ given in Proposition \ref{proposition:3.10}, we have $H<\frac{\sqrt{2}}{2}<1$, then we obtain
		\begin{equation*}
			\dot{\phi}  < -12 \phi   + \sqrt{3}, \quad \text{for all } t\in (\mathcal{T}_-,0).
		\end{equation*}
Define a comparison function $\underline{\phi}_{-}(t)$ as the solution to the ODE
		\begin{equation*}
			\dot{\underline{\phi}}_-(t) = -12 \underline{\phi}_-(t)  + \sqrt{3}, \quad \text{with} \quad \underline{\phi}_-(0) = 0. 
		\end{equation*} 
		
	This equation can be solved explicitly via the integrating factor method, yielding
		\begin{equation}
			\underline{\phi}_-(t) =\frac{\sqrt{3}}{12}(1-e^{-12t}) . 
            \label{eq:underlinephi-}
		\end{equation}  
	Considering Theorem \ref{theorem:C}, we conclude that $\phi(t)>\underline{\phi}_-(t)$ for all  $(\mathcal{T}_-,0)$.
	\end{proof}

\subsubsection{Upper bound of $\phi$  for $t \in (0, \mathcal{T}_+)$}
In this subsection, we establish an upper bound for $\phi$ on the interval $t \in (0, \mathcal{T}_+)$.

\begin{proposition}
	\label{proposition:3.14++++}
	Under the initial conditions \eqref{eq:2.12}--\eqref{eq:2.11}, or alternatively \eqref{eq:1.13++++}--\eqref{1.14++++}, $\phi$ satisfies
	\begin{equation*}
		\phi(t) < \sqrt{3} \ln(\beta t + 1)+\alpha, \quad \text{for all }  t \in (0, \, \mathcal{T}_+).
	\end{equation*}
    We note that $\alpha=0$ for the data \eqref{eq:2.12}. 
\end{proposition}

\begin{proof}
	Applying \eqref{eq:B2} to \eqref{eq:3.4}, with the help of  the facts that $H > 0$ and $\phi > 0$ on $(0, \mathcal{T}_+)$,   it follows that
	\begin{equation}
		\dot{\phi} < \sqrt{3}H \overset{\text{Prop. \ref{proposition:3.5}}}{<}\frac{\sqrt{3}}{t + \frac{1}{\beta}}, \quad \text{for all } t\in (0, \, \mathcal{T}_+),
\label{eq3.14}
	\end{equation}  
where we note that Proposition \ref{proposition:3.5} is guaranteed by initial conditions \eqref{eq:2.12}--\eqref{eq:2.11} or \eqref{eq:1.13++++}--\eqref{1.14++++}.

	Define a comparison function $\overline{\phi}_+$ as the solution to the ODE
	\begin{equation}
		\dot{\overline{\phi}}_+(t) = \frac{\sqrt{3}}{t + \frac{1}{\beta}}, \quad \text{with} \quad 	\overline{\phi}_+(0) = \alpha \implies \overline{\phi}_+(t) = \sqrt{3} \ln(\beta t + 1)+\alpha.
        \label{eq:barphi+}
	\end{equation}  
We remark that under the initial condition \eqref{eq:2.12} given in Theorem \ref{theorem:2.1}, $\alpha=0$.
\end{proof}

\subsubsection{Upper bound of $\phi$  for $t \in (\mathcal{T}_-,0)$}
Finally, we derive an upper bound for $\phi$ on the interval $t \in (\mathcal{T}_-, 0)$.
 
\begin{proposition}
	\label{proposition:4.5++}
	Under the initial conditions \eqref{eq:2.11}--\eqref{eq:2.13}, $\phi$ satisfies
	\begin{align*}
		\phi(t)< \mathfrak{H}(t):=\begin{cases}
			\displaystyle	\frac{\sqrt{3}}{6\beta^2}(1-e^{-6\beta^3t}) <0, \quad &\text{if }0< \beta \leq \sqrt{\frac{5}{27}} , \\
			\displaystyle	\frac{25\sqrt{3}}{24}(1-e^{-\frac{48}{125}t})<0 , \quad &\text{if }\sqrt{\frac{5}{27}}< \beta < \frac{\sqrt{3}}{3},
		\end{cases}
	\end{align*}
   where $t \in (\mathcal{T}_-, 0)$.
\end{proposition}

\begin{proof}
From \eqref{eq:3.4}, \eqref{eq:B3}, and the facts that $H > 0$ and $\phi < 0$ on   $(\mathcal{T}_-,0)$, we obtain
	\begin{equation}
		\dot{\phi}>\sqrt{3}H-6H^3\phi \overset{\text{Prop. \ref{t:Hlrbd2}}}{>}\begin{cases}
			\displaystyle	\sqrt{3}\beta-6\beta^3\phi, \quad &\text{if }0< \beta \leq \sqrt{\frac{5}{27}} , \\
			\displaystyle	\frac{2\sqrt{3}}{5}-\frac{48}{125}\phi, \quad &\text{if }\sqrt{\frac{5}{27}}< \beta < \frac{\sqrt{3}}{3},
		\end{cases} \quad \text{for all } t\in (\mathcal{T}_-,0).
        \label{eq4.5}
	\end{equation} 
	Define $\bar \phi_-(t)$ as the solution to the ODE
	\begin{equation*}
		\dot{\bar \phi}_- =\begin{cases}
			\displaystyle	\sqrt{3}\beta-6\beta^3{\bar \phi}_ , \quad &\text{if }0< \beta \leq \sqrt{\frac{5}{27}} , \\
			\displaystyle	\frac{2\sqrt{3}}{5}-\frac{48}{125}\bar \phi_- , \quad &\text{if }\sqrt{\frac{5}{27}}< \beta < \frac{\sqrt{3}}{3}  ,
		\end{cases} \quad \text{with} \quad  	\bar \phi_-(0)=0.
	\end{equation*}  
This equation can be solved explicitly, yielding $\bar \phi_-(t)= \mathfrak{H}(t)$. 
Theorem \ref{theorem:C} leads to $\phi(t)<\bar \phi_-(t)$ on $(\mathcal{T}_-,0)$.
\end{proof}

\section{Proof of the Main Theorem}\label{sec:4}
Now we are in a position to prove the main theorems, Theorems \ref{theorem:2.1} and \ref{theorem:1.2}.

\begin{proof}[Proof of Theorem \ref{theorem:2.1} and \ref{theorem:1.2}] 
We have established the local existence of solutions to the system \eqref{eq:2.4++++++}--\eqref{eq:2.3++++++} (see Proposition \ref{t:locext}), along with  bounds on  $H$ and $\phi$ over the maximal interval of existence $(\mathcal{T}_-,\mathcal{T}_+)$ (see Propositions \ref{proposition:3.4},  \ref{proposition:3.5}, \ref{proposition:3.10}, \ref{t:Hlrbd2}, \ref{proposition:4.2}, \ref{proposition:3.13++++}, \ref{proposition:3.14++++} and \ref{proposition:4.5++}).  
Recalling the estimates for $\phi$ and $H$ from these propositions, we obtain the following bounds
\begin{gather}
	0< \frac{1}{5t+\frac{1}{\beta}} < H(t) < \frac{1}{t+\frac{1}{\beta}} < \beta,   \label{eq:5.1}
	\\
	0<	\frac{-1+\sqrt{1+\frac{24\beta^2}{5}\ln{(1+5\beta t)}}}{4\sqrt{3}\beta^2} < \phi(t) < \sqrt{3} \ln(\beta t+1)<\sqrt{3} \ln(\beta \mathcal{T}_+ +1),      \label{eq:5.2}
\end{gather}
for all $t\in (0,\mathcal{T}_+)$, and 
\begin{gather}
	0<	\mathfrak{L}(t)< H(t) <\frac{2 \beta +\sqrt{2}-\left(\sqrt{2}-2 \beta \right) e^{3 \sqrt{2} t}}{\sqrt{2} \left(2 \beta  +\sqrt{2}+\left(\sqrt{2}-2 \beta \right) e^{3 \sqrt{2} t}\right)}<\frac{1}{\sqrt{2}},   \label{eq:5.3++} \\
	\frac{\sqrt{3}}{12}(1-e^{-12 \mathcal{T}_-})  <	\frac{\sqrt{3}}{12}(1-e^{-12t})  < \phi(t) <  \mathfrak{H}(t)<0,  
	\label{eq:5.4++}
\end{gather}
for all $t\in (\mathcal{T}_-,0)$, 
where $\mathfrak{L}$ and $\mathfrak{H}$ are defined in Proposition  \ref{t:Hlrbd2} and \ref{proposition:4.5++}, respectively.

We now \textit{aim} to show that this interval is global-in-time, that is,   $\mathcal{T}_- = -\infty$ and $\mathcal{T}_+ = +\infty$.   This will complete the proofs of Theorems~\ref{theorem:2.1} and \ref{theorem:1.2}. To this end, we focus on Theorem~\ref{theorem:2.1}; the proof of Theorem~\ref{theorem:1.2} follows similarly with minor modifications, and we omit the details.

Before establishing the global-in-time result, we first derive estimates for $\dot{\phi}$ and the scale factor $a$:

\underline{Estimates of $\dot{\phi}$:}  	
From \eqref{eq3.13} in the proof of Proposition~\ref{proposition:3.13++++} and  \eqref{eq4.5} in the proof of Proposition~\ref{proposition:4.5++}, we have 
\begin{equation*}
	\sqrt{3}H-6H^3\phi<\dot{\phi}<\sqrt{3}H-12H^3\phi,
\end{equation*}
 for all  $ t \in (\mathcal{T}_-,0)$. 
In this interval, Proposition \ref{proposition:3.10} provides the constant upper bound for $H$, Proposition \ref{t:Hlrbd2} implies the improved lower bound for $H$. Bounds for $\phi$ are given by Propositions \ref{proposition:3.13++++} and \ref{proposition:4.5++}. Substituting all the four bounds into the inequality yields
\begin{equation}
0<	\mathfrak{W}(t)<\dot{\phi}<\sqrt{3}e^{-12t} < \sqrt{3}e^{-12 \mathcal{T}_- } ,\quad \text{for all } t\in (\mathcal{T}_-,0), \label{phi-}
\end{equation}
where $\mathfrak{W}(t)$ is defined by
\begin{equation*}
	\mathfrak{W}(t):=\begin{cases}
		\displaystyle	\sqrt{3}\beta e^{-6\beta^3t} , \quad &\text{if }0< \beta \leq \sqrt{\frac{5}{27}} , \\
		\displaystyle	\frac{2\sqrt{3}}{5}e^{-\frac{48}{125}t} , \quad &\text{if } \sqrt{\frac{5}{27}}< \beta < \frac{\sqrt{3}}{3}  .
	\end{cases} 
\end{equation*}

On the other hand, for all  $ t \in (0,\mathcal{T}_+)$, by \eqref{eq4.2} in the proof of Proposition \ref{proposition:4.2} and \eqref{eq3.14} in the proof of Proposition \ref{proposition:3.14++++}, we have
\begin{equation*}
	\frac{\sqrt{3}H}{1+4\sqrt{3}\phi H^2}<\dot{\phi}<\sqrt{3}H,
\end{equation*}
for all $ t \in (0,\mathcal{T}_+)$.  In this interval, Propositions~\ref{proposition:3.4} and \ref{proposition:3.5} provide lower and upper bounds for $H$, while Propositions~\ref{proposition:4.2} and \ref{proposition:3.14++++} provide bounds for $\phi$. Substituting these bounds into the inequality yields
\begin{equation}
0<	\frac{1}{\sqrt{1+12\beta^2\ln{(\beta t+1)}}}\frac{\sqrt{3}}{5t+\frac{1}{\beta}}<\dot{\phi}<\frac{\sqrt{3}}{t+\frac{1}{\beta}} < \sqrt{3} \beta ,\quad \text{for all } t\in (0,\mathcal{T}_+)  .  
	\label{phi+}
\end{equation}

\underline{Estimates of $a$:} We now derive estimates for the scale factor  $a$, which is given by $a(t):=a_0 e^{\int^t_0H(s)ds}$. Directly integrating \eqref{eq:5.1} and \eqref{eq:5.3++} yields
 \begin{equation}
	0<	a_0 \left( \frac{2}{(\sqrt{2}\beta + 1) + (1 - \sqrt{2}\beta) e^{3\sqrt{2} t}} \right)^{\frac{1}{3}} e^ { \frac{ \sqrt{2} t}{2 } }<a(t)<a_0 \mathfrak{G}(t) <a_0, \quad \text{ for all } t \in (\mathcal{T}_-,0),
		\label{eq:a-}
	\end{equation}
	where $\mathfrak{G}(t)$ is defined by
	\begin{equation*}
	\mathfrak{G}(t)=  \begin{cases} 
		\left( \dfrac{227}{225 e^ { -\frac{454 \beta t}{45}}  + 2} \right)^{\frac{1}{10}} < 1 , \quad &\text{if }0< \beta \leq \sqrt{\frac{5}{27}},\\ 
		\biggl( \dfrac{2}{5\beta e^{-4t}-(5\beta-2)}\biggr)^{\frac{1}{10}} <1
		, \quad &\text{if } \sqrt{\frac{5}{27}}< \beta < \frac{\sqrt{3}}{3}.
	\end{cases}
\end{equation*}
For $t \in (0, \mathcal{T}_+)$,  we similarly obtain the bounds
	\begin{equation}
		a_0<a_0 \left( 5\beta t + 1 \right)^{\frac{1}{5}}<a(t)< a_0 (\beta t + 1) < a_0 (\beta \mathcal{T}_+ + 1)  \quad \text{ for all } t \in (0, \mathcal{T}_+).
		\label{eq:a+}
	\end{equation}

\underline{Extensions of solutions:} Let us prove $\mathcal{T}_- = -\infty$ and $\mathcal{T}_+ = +\infty$ by contradiction. 	
We first assume, for contradiction, that $\mathcal{T}_+ < +\infty$, and focus on the solution on $(0,\mathcal{T}_+)$. 
From the bounds established in \eqref{eq:5.1}--\eqref{eq:a+}, we conclude that there exists a constant
\begin{equation*}
R>\max\left\{  \sqrt{3} \beta,   a_0 (\beta \mathcal{T}_+ + 1), \beta,  \sqrt{3} \ln(\beta \mathcal{T}_+ +1)  \right\}>0 , 
\end{equation*}
such that $\mathcal{U}:=\p{a,H,\phi,\Phi}^T\in B_R(0) \subset \Rbb^4$ for all $t\in(0,\mathcal{T}_+)$, where $B_R(0)$ denotes the open ball centered at the origin with radius  $R$.

Recall that the system \eqref{e:locsys} can be written as an ODE:
\begin{equation*}
	\frac{d}{dt} \mathcal{U}= \mathcal{F}(\mathcal{U})
\end{equation*}
where  
\begin{equation*}
 \mathcal{F}(\mathcal{U}):= \p{ a H\\	F_1(H,\phi,\Phi) \\
 	\Phi \\ F_2(H,\phi,\Phi)  }
\end{equation*}
and $F_1$ and $F_2$ are defined in \eqref{e:F1} and \eqref{e:F2}, respectively. Since  $\mathcal{F}\in C^1(\overline{B_R(0)}, \Rbb^4)$, it is Lipschitz continuous and bounded on $\overline{B_R(0)}$. 
Considering Corollary~\ref{t:contthm2}, the solution $\mathcal{U}$ can therefore be  continued to the right passing through the point $\mathcal{T}_+$. This contradicts the assumption that $\mathcal{T}_+$ is the maximal time of existence.  Hence, we conclude that $\mathcal{T}_+ = +\infty$. A similar argument applies to the past. Focusing on the interval $(\mathcal{T}_-,0)$, and choosing
 \begin{equation*}
 	R>\max\left\{ \sqrt{3}e^{-12 \mathcal{T}_- },   a_0,    \frac{1}{\sqrt{2}}, 	\frac{\sqrt{3}}{12}(e^{-12 \mathcal{T}_-}-1)  \right\}>0,
 \end{equation*}
and applying the same continuation argument to conclude that $\mathcal{T}_-=-\infty$. We complete the proof.   
\end{proof}

	\appendix

	\section{Basic ode theorems}
	\label{app:A}
    In this appendix, we state the basic existence and comparison theorems that will be used throughout this article. We omit the proofs, which can be found in standard ODE references such as \cite{Hsu2013,you1982}. 
	\begin{theorem}[Picard-Lindelöf Theorem]
		\label{theorem:PL}
		Let $ D \subseteq \mathbb{R} \times \mathbb{R}^n $ be a closed domain containing the point $ (t_0, y_0) $, and let $ f: D \to \mathbb{R}^n $ be a continuous function satisfying  the Lipschitz conditions in $y$ with Lipschitz constant $L$. Then, the initial value problem
        \begin{equation}\label{e:ode1}
			\frac{dy}{dt} = f(t, y), \quad y(t_0) = y_0
		\end{equation}
		has a unique $C^1$ solution $ y(t) $ defined on some interval $ [t_0 - h, t_0 + h] $ for some $ h > 0 $.
	\end{theorem} 
	\begin{theorem}[Comparison Theorem]
		\label{theorem:C}
Let $ f(t,x) $ and $ F(t,x) $ be two continuous scalar functions defined on a planar region $ D $,  such that
\begin{equation*}
f(t,x) < F(t,x), \quad (t,x) \in D. 
\end{equation*}
If $ x = \varphi(t) $ and $ x = \Phi(t) $ are solutions of the differential equations
\begin{equation*}
x' = f(t,x) \quad \text{and} \quad x' = F(t,x), 
\end{equation*}
respectively, passing through the same point $ (\tau, \xi) \in D $, then the following conclusion holds:
\begin{enumerate}[leftmargin=*]
    \item[(1)] $ \varphi(t) < \Phi(t) $ for $ t > \tau $ and $ t $ in the common interval of existence;
    \item[(2)] $ \varphi(t) > \Phi(t) $ for $ t < \tau $ and $ t $ in the common interval of existence.
\end{enumerate}
	\end{theorem}

    \begin{theorem}[Continuation of solutions]\label{t:contthm1}
		Let $f\in C(D)$ and satisfy $|f(t,y)|\leq M$ for some constant $M>0$ and $(t,y) \in D$. Suppose $\phi$ is a solution of \eqref{e:ode1} on the interval $J=(a,b)$. Then
		
		$(1)$ $\lim_{t\rightarrow a+} \phi(t)=\phi(a)$ and $\lim_{t\rightarrow b-} \phi(t)=\phi(b)$ both exist and finite.
		
		$(2)$ if $(b,\phi(b)) \in D$, then the solution $\phi$ can be continued to the right passing through the point $t=b$.
	\end{theorem}
	
	\begin{corollary}[Continuation principle]\label{t:contthm2}
		Let $f\in C(D)$. Suppose $\phi$ is a solution of \eqref{e:ode1} on the interval $J=(a,b)$, and if there is a finite constant $M>0$, such that for every $t\in (a,b)$,
		\begin{equation*}
			|f(t,\phi(t))|\leq M<+\infty,
		\end{equation*}
		then the solution $\phi$ can be continued to the right passing through the point $t=b$.
	\end{corollary}

	\section{Useful inequalities}
	\label{app:B}
	\begin{lemma}
		For all positive $x>0$,  the following inequalities hold
        	\begin{equation}
			\sqrt{1+x^2}<1+x,
			\label{eq:B2}
		\end{equation}
		\begin{equation}
			\sqrt{1+x^2}-x>\frac{1}{1+2x},
			\label{eq:B1}
		\end{equation}
		\begin{equation}
			\sqrt{1+x^2}>1.
			\label{eq:B3}
		\end{equation}
	\end{lemma}
	\begin{proof}
		The inequalities \eqref{eq:B2} and \eqref{eq:B3} follow immediately by squaring both sides and simplifying. To prove \eqref{eq:B1}, we note
        \begin{equation*}
        \sqrt{1+x^2}-x=\frac{1}{\sqrt{1+x^2}+1} \stackrel{\text{\eqref{eq:B2}}}{>}\frac{1}{1+2x}.
        \end{equation*}
        It concludes \eqref{eq:B1}.
	\end{proof}

 \begin{lemma}\label{t:pre1}
 	If $\beta<x<1$ and $0<\beta \leq \frac{1}{3} \sqrt{\frac{5}{3}}$, then 
 	$2 x^2+\frac{454 \beta }{135 x}-\frac{56}{15}\leq 0$. 
 \end{lemma}
 \begin{proof}
Define a function $f(x):= 2x^2 + \frac{454\beta}{135x} - \frac{56}{15}$. 
We aim to show that \(f(x) < 0\) on the interval \(x \in (\beta, 1)\), assuming \(0<\beta \leq \frac{1}{3} \sqrt{\frac{5}{3}}\). It is straightforward to verify that $f(x)$ attains a minimum at $x_\text{crit}=(\frac{227 \beta }{270})^{\frac{1}{3}}$ by analyzing $f'(x)$. Therefore, $f(x)$ is strictly decreasing on $(\beta, x_{\text{crit}})$, strictly increasing on $(x_{\text{crit}}, 1)$.  Since $0<\beta \leq \frac{1}{3} \sqrt{\frac{5}{3}}$, we can verify $x_\text{crit}\in(\beta,1)$. Therefore, the maximum value of \(f(x)\) on \((\beta, 1)\) is achieved at the endpoints:  $f(1)=\frac{454 \beta }{135}-\frac{26}{15}<\frac{454}{81 \sqrt{15}}-\frac{26}{15}<0$ and $f(\beta)=2 \beta ^2-\frac{10}{27}\leq 0$. 
Thus, the lemma is proved.
 \end{proof}

 \begin{lemma}\label{t:pre2}
 If $\frac{2}{5}<x<1$, then $2 x^2+\frac{4}{3 x}-\frac{100}{27}<0$. 
 \end{lemma}
 \begin{proof}
 	 The proof follows from a similar strategy to that of Lemma~\ref{t:pre1}.  Let $f(x):=2 x^2+\frac{4}{3 x}-\frac{100}{27}$,  and thus its  minimum point is $x_\text{crit}=(\frac{1}{3})^{\frac{1}{3}}\in(\frac{2}{5},1)$.  the maximum of $f(x)$ on the interval $(\frac{2}{5}, 1)$ must occur at the endpoints:  $f(\frac{2}{5})=-\frac{34}{675}<0 $ and $f(1)=-\frac{10}{27}<0$. Thus we complete the proof. 
 \end{proof}

\section*{Acknowledgements}

C.L. is partially supported by NSFC (Grant No. $12571234$) and Fundamental Research Support Program of
HUST (Grant No. $2025$BRSXA$001$). J.W. is partially supported by NSFC (Grant No. $12271450$).

\bigskip

\textbf{Data Availability} Data sharing is not applicable to this article as no datasets were
generated or analysed during the current study.

\bigskip

\textbf{Declarations}

\bigskip

\textbf{Conflict of interest} The authors declare that they have no conflict of interest.

	\bibliographystyle{amsplain}
	\bibliography{ref}

@Article{Myung2018,
  author    = {Myung and Yun Soo and Zou, De-Cheng},
  journal   = {Physical Review D},
  title     = {{Gregory-Laflamme instability of black hole in Einstein-scalar-Gauss-Bonnet theories}},
  year      = {2018},
  month     = {Jul},
  pages     = {024030},
  volume    = {98},
  doi       = {10.1103/PhysRevD.98.024030},
  issue     = {2},
  numpages  = {6},
  publisher = {American Physical Society},
  url       = {https://link.aps.org/doi/10.1103/PhysRevD.98.024030},
}

@Article{Hikmawan2016,
  author    = {Getbogi Hikmawan and Jiro Soda and Agus Suroso and Freddy Permana Zen},
  journal   = {Physical Review D},
  title     = {{Comment on ''Gauss-Bonnet inflation''}},
  year      = {2016},
  month     = {Mar},
  pages     = {068301},
  volume    = {93},
  doi       = {10.1103/PhysRevD.93.068301},
  issue     = {6},
  numpages  = {5},
  publisher = {American Physical Society},
  url       = {https://link.aps.org/doi/10.1103/PhysRevD.93.068301},
}

@Article{Kanti2015,
  author    = {Panagiota Kanti and Radouane Gannouji and Naresh Dadhich},
  journal   = {Physical Review D},
  title     = {{Early-time cosmological solutions in Einstein-scalar-Gauss-Bonnet theory}},
  year      = {2015},
  month     = {Oct},
  pages     = {083524},
  volume    = {92},
  doi       = {10.1103/PhysRevD.92.083524},
  issue     = {8},
  numpages  = {15},
  publisher = {American Physical Society},
  url       = {https://link.aps.org/doi/10.1103/PhysRevD.92.083524},
}

@article{borde2003inflationary,
  title={Inflationary spacetimes are incomplete in past directions},
  author={Arvind Borde and Alan H Guth and Alexander Vilenkin},
  journal={Physical review letters},
  volume={90},
  number={15},
  pages={151301},
  year={2003},
  publisher={APS}
}

@Article{Kanti2015a,
  author    = {Panagiota Kanti and Radouane Gannouji and Naresh Dadhich},
  journal   = {Physical Review D},
  title     = {{Gauss-Bonnet inflation}},
  year      = {2015},
  month     = {Aug},
  pages     = {041302},
  volume    = {92},
  doi       = {10.1103/PhysRevD.92.041302},
  issue     = {4},
  numpages  = {5},
  publisher = {American Physical Society},
  url       = {https://link.aps.org/doi/10.1103/PhysRevD.92.041302},
}

@Article{Sberna2017,
  author    = {Laura Sberna and Paolo Pani},
  journal   = {Physical Review D},
  title     = {{Nonsingular solutions and instabilities in Einstein-scalar-Gauss-Bonnet cosmology}},
  year      = {2017},
  month     = {Dec},
  pages     = {124022},
  volume    = {96},
  doi       = {10.1103/PhysRevD.96.124022},
  issue     = {12},
  numpages  = {9},
  publisher = {American Physical Society},
  url       = {https://link.aps.org/doi/10.1103/PhysRevD.96.124022},
}

@Article{Kanti1999,
  author    = {Panagiota Kanti and John Rizos and Kyriakos Tamvakis},
  journal   = {Physical Review D},
  title     = {{Singularity-free cosmological solutions in quadratic gravity}},
  year      = {1999},
  month     = {Mar},
  pages     = {083512},
  volume    = {59},
  doi       = {10.1103/PhysRevD.59.083512},
  issue     = {8},
  numpages  = {12},
  publisher = {American Physical Society},
  url       = {https://link.aps.org/doi/10.1103/PhysRevD.59.083512},
}

@Article{Rizos1994,
  author   = {John Rizos and Kyriakos Tamvakis},
  journal  = {Physics Letters B},
  title    = {{On the existence of singularity-free solutions in quadratic gravity}},
  year     = {1994},
  issn     = {0370-2693},
  number   = {1},
  pages    = {57-61},
  volume   = {326},
  doi      = {https://doi.org/10.1016/0370-2693(94)91192-4},
  url      = {https://www.sciencedirect.com/science/article/pii/0370269394911924},
}

@Article{Kawai1998,
  author   = {Shinsuke Kawai and Masa-aki Sakagami and Jiro Soda},
  journal  = {Physics Letters B},
  title    = {Instability of 1-loop superstring cosmology},
  year     = {1998},
  issn     = {0370-2693},
  number   = {3},
  pages    = {284-290},
  volume   = {437},
  doi      = {https://doi.org/10.1016/S0370-2693(98)00925-3},
  url      = {https://www.sciencedirect.com/science/article/pii/S0370269398009253},
}

@InProceedings{Kawai1999,
  author        = {Shinsuke Kawai and Masa-aki Sakagami and Jiro Soda},
  booktitle     = {Proceedings, 7th Workshop on General Relativity and Gravitation (JGRG7) : Kyoto, Japan, October 27-30, 1997},
  title         = {Perturbative analysis of non-singular cosmological model},
  year          = {1999},
  archiveprefix = {arXiv},
  eprint        = {gr-qc/9901065},
  primaryclass  = {gr-qc},
  url           = {https://arxiv.org/abs/gr-qc/9901065},
}

@Article{Antoniou2018,
  author    = {Georgios Antoniou and Athanasios Bakopoulos and Panagiota Kanti},
  journal   = {Physical Review Letters.},
  title     = {Evasion of No-Hair Theorems and Novel Black-Bole Solutions in {Gauss-Bonnet} Theories},
  year      = {2018},
  month     = {Mar},
  pages     = {131102},
  volume    = {120},
  doi       = {10.1103/PhysRevLett.120.131102},
  issue     = {13},
  numpages  = {6},
  publisher = {American Physical Society},
  url       = {https://link.aps.org/doi/10.1103/PhysRevLett.120.131102},
}

@Article{Doneva2018,
  author    = {Daniela Doneva and Stoytcho Yazadjiev},
  journal   = {Physical Review Letters.},
  title     = {New{ Gauss-Bonnet} Black Holes with Curvature-Induced Scalarization in Extended Scalar-Tensor Theories},
  year      = {2018},
  month     = {Mar},
  pages     = {131103},
  volume    = {120},
  doi       = {10.1103/PhysRevLett.120.131103},
  issue     = {13},
  numpages  = {6},
  publisher = {American Physical Society},
  url       = {https://link.aps.org/doi/10.1103/PhysRevLett.120.131103},
}

@Article{Silva2018,
  author    = {Hector O. Silva and Jeremy Sakstein and Leonardo Gualtieri and Thomas P. Sotiriou and Emanuele Berti},
  journal   = {Physical Review Letters.},
  title     = {{Spontaneous scalarization of black holes and compact stars from a Gauss-Bonnet coupling}},
  year      = {2018},
  month     = {Mar},
  pages     = {131104},
  volume    = {120},
  doi       = {10.1103/PhysRevLett.120.131104},
  issue     = {13},
  numpages  = {6},
  publisher = {American Physical Society},
  url       = {https://link.aps.org/doi/10.1103/PhysRevLett.120.131104},
}

@MastersThesis{Sberna2017a,
  author        = {Laura Sberna},
  school        = {University of Rome},
  title         = {{Early-universe cosmology in Einstein-scalar-Gauss-Bonnet gravity}},
  year          = {2017},
  archiveprefix = {arXiv},
  eprint        = {1708.01150},
  primaryclass  = {gr-qc},
  url           = {https://arxiv.org/abs/1708.01150},
}

@Article{Starobinsky1980,
  author   = {Alexei Starobinsky},
  journal  = {Physics Letters B},
  title    = {{A new type of isotropic cosmological models without singularity}},
  year     = {1980},
  issn     = {0370-2693},
  number   = {1},
  pages    = {99-102},
  volume   = {91},
  doi      = {https://doi.org/10.1016/0370-2693(80)90670-X},
  url      = {https://www.sciencedirect.com/science/article/pii/037026938090670X},
}

@Article{Trodden1993,
  author   = {Mark Trodden and Viatcheslav F Mukhanov  and Robert H. Brandenberger
},
  journal  = {Physics Letters B},
  title    = {{A nonsingular two dimensional black hole}},
  year     = {1993},
  issn     = {0370-2693},
  number   = {4},
  pages    = {483-487},
  volume   = {316},
  doi      = {https://doi.org/10.1016/0370-2693(93)91032-I},
  url      = {https://www.sciencedirect.com/science/article/pii/037026939391032I},
}

@Article{Brandenberger1993,
  author    = {Robert Brandenberger and Viatcheslav F Mukhanov and Andrew Sornborger},
  journal   = {Physical Review D},
  title     = {{Cosmological theory without singularities}},
  year      = {1993},
  month     = {Aug},
  pages     = {1629--1642},
  volume    = {48},
  doi       = {10.1103/PhysRevD.48.1629},
  issue     = {4},
  numpages  = {0},
  publisher = {American Physical Society},
  url       = {https://link.aps.org/doi/10.1103/PhysRevD.48.1629},
}

@Book{Hsu2013,
  author    = {Sze-Bi Hsu},
  publisher = {World scientific},
  title     = {Ordinary Differential Equations with Applications},
  year      = {2013},
  month     = {jan},
  doi       = {10.1142/8744},
}

@Book{ChoquetBruhat2009,
  author    = {Yvonne Choquet-Bruhat},
  publisher = {Oxford University Press},
  title     = {{General relativity and the Einstein equations}},
  year      = {2009},
}

@book{Ringstroem2009,
	Author = {Ringstr{\"o}m, Hans},
	Isbn = {9783037190531},
	Lccn = {2009281497},
	Publisher = {European Mathematical Society},
	Series = {ESI lectures in mathematics and physics},
	Title = {The {Cauchy} Problem in General Relativity},
	Url = {https://books.google.com.au/books?id=Bn\_cC7QwQ0MC},
	Year = {2009}}

@book{Hawking2010,
	Author = {Stephen W. Hawking and Ellis, George F. R.},
	Ean = {9780521099066},
	Isbn = {0521099064},
	Month = oct,
	Pagetotal = {404},
	Publisher = {Cambridge University Press},
	Title = {The Large Scale Structure of Space-Time},
	Url = {https://www.ebook.de/de/product/3238810/stephen_hawking_the_large_scale_structure_of_space_time.html},
	Year = {2010}}

@Article{Liu2022,
  author    = {Chao Liu and Yiqing Shi},
  journal   = {Physical Review D},
  title     = {Rigorous proof of the slightly nonlinear {Jeans} instability in the expanding {Newtonian} universe},
  year      = {2022},
  month     = {feb},
  number    = {4},
  pages     = {043519},
  volume    = {105},
  doi       = {10.1103/physrevd.105.043519},
  publisher = {American Physical Society ({APS})},
}

@article{Ellis2003,
doi = {10.1088/0264-9381/21/1/015},
url = {https://dx.doi.org/10.1088/0264-9381/21/1/015},
year = {2003},
month = {nov},
publisher = {},
volume = {21},
number = {1},
pages = {223},
author = {George F. R. Ellis and Roy Maartens},
title = {{The emergent universe: inflationary cosmology with no singularity}},
journal = {Classical and Quantum Gravity}
}

@article{Hawking2005,
doi = {10.1238/Physica.Topical.117a00049},
url = {https://dx.doi.org/10.1238/Physica.Topical.117a00049},
year = {2005},
month = {jan},
publisher = {},
volume = {2005},
number = {T117},
pages = {49},
author = {Stephen W. Hawking},
title = {A Non Singular Universe},
journal = {Physica Scripta}
}

@Article{Liu2022b,
  author    = {Chao Liu},
  journal   = {Mathematische Annalen},
  title     = {Blowups for a class of second order nonlinear hyperbolic equations: a reduced model of nonlinear Jeans instability},
  year      = {2025},
  date      = {2025/09/01},
  volume    = {393},
  number    = {1},
  pages     = {317--363},
  issn      = {1432-1807},
  doi       = {10.1007/s00208-025-03260-0},
  url       = {},
}

@Article{Liu2023a,
  author  = {Chao Liu},
  journal = {Physical Review D},
  title   = {Fully nonlinear gravitational instabilities for expanding {Newtonian} universes with inhomogeneous pressure and entropy: beyond the {T}olman's solution},
  year    = {2023},
  month   = jun,
  number  = {12},
  pages   = {123534},
  volume  = {107},
  doi     = {10.1103/PhysRevD.107.123534},
  eprint  = {https://arxiv.org/abs/2210.04657#},
  url     = {https://journals.aps.org/prd/abstract/10.1103/PhysRevD.107.123534},
}

@Article{Liu2023b,
  author        = {Chao Liu},
  journal       = {arXiv:2305.13211},
  title         = {{Fully nonlinear gravitational instabilities for expanding spherical symmetric Newtonian universes with inhomogeneous density and pressure}},
  year          = {2023},
  month         = may,
  archiveprefix = {arXiv},
  eprint        = {2305.13211},
  primaryclass  = {math-ph},
}

@Article{Liu2024,
  author        = {Chao Liu},
  journal       = {arXiv:2409.02516},
  title         = {{The emergence of nonlinear Jeans-type instabilities for quasilinear wave equations}},
  year          = {2024},
  month         = sep,
  archiveprefix = {arXiv},
  copyright     = {arXiv.org perpetual, non-exclusive license},
  doi           = {10.48550/ARXIV.2409.02516},
  eprint        = {2409.02516},
  primaryclass  = {math.AP},
  publisher     = {arXiv},
}

@article{HawkingPenrose1965,
  title = {Gravitational Collapse and Space-Time Singularities},
  author = {Roger Penrose},
  journal = {Physical Review Letters.},
  volume = {14},
  issue = {3},
  pages = {57--59},
  numpages = {0},
  year = {1965},
  month = {Jan},
  publisher = {American Physical Society},
  doi = {10.1103/PhysRevLett.14.57},
  url = {https://link.aps.org/doi/10.1103/PhysRevLett.14.57}
}

@article{HawkingPenrose1970,
    author = {Stephen W. Hawking and Roger Penrose},
    title = {The singularities of gravitational collapse and cosmology},
    journal = {Proceedings of the Royal Society of London. A. Mathematical and Physical Sciences},
    volume = {314},
    number = {1519},
    pages = {529-548},
    year = {1970},
    month = {01},
    issn = {0080-4630},
    doi = {10.1098/rspa.1970.0021},
    url = {https://doi.org/10.1098/rspa.1970.0021},
    eprint = {https://royalsocietypublishing.org/rspa/article-pdf/314/1519/529/58574/rspa.1970.0021.pdf},
}

@article{GrossSloan1987,
    author = "David Jonathan Gross and John Howard Sloan IV",
    title = {The Quartic Effective Action for the Heterotic String},
    reportNumber = "NSF-ITP-87-02",
    doi = "10.1016/0550-3213(87)90465-2",
    journal = "Nuclear Physics B",
    volume = "291",
    pages = "41--89",
    year = "1987"
}

@Article{he2025,
  author        = {Chihang He and Chao Liu},
  journal       = {arXiv:2512.00455},
  title         = {Existence and bounds of nonlinear singularity-free cosmological solutions in a string-inspired gravity},
  year          = {2025},
  archiveprefix = {arXiv},
  eprint        = {2512.00455},
  primaryclass  = {math.AP},
  url           = {https://arxiv.org/abs/2512.00455},
}

@article{Zwiebach1985,
title = {Curvature squared terms and string theories},
journal = {Physics Letters B},
volume = {156},
number = {5},
pages = {315-317},
year = {1985},
issn = {0370-2693},
doi = {https://doi.org/10.1016/0370-2693(85)91616-8},
url = {https://www.sciencedirect.com/science/article/pii/0370269385916168},
author = {Barton Zwiebach}
}

@Article{Damour1993,
  author    = {Thibault Damour and Gilles Esposito-Farèse
},
  journal   = {Physical Review Letters.},
  title     = {{Nonperturbative strong-field effects in tensor-scalar theories of gravitation}},
  year      = {1993},
  month     = {Apr},
  pages     = {2220--2223},
  volume    = {70},
  doi       = {10.1103/PhysRevLett.70.2220},
  issue     = {15},
  numpages  = {0},
  publisher = {American Physical Society},
  url       = {https://link.aps.org/doi/10.1103/PhysRevLett.70.2220},
}

@Article{Damour1996,
  author    = {Thibault Damour and Gilles Esposito-Farèse
},
  journal   = {Physical Review D},
  title     = {{Tensor-scalar gravity and binary-pulsar experiments}},
  year      = {1996},
  month     = {Jul},
  pages     = {1474--1491},
  volume    = {54},
  doi       = {10.1103/PhysRevD.54.1474},
  issue     = {2},
  numpages  = {0},
  publisher = {American Physical Society},
  url       = {https://link.aps.org/doi/10.1103/PhysRevD.54.1474},
}

@Book{you1982,
  author    = {You, Bingli},
  publisher = {Higher Education Press},
  title     = {Supplementary Course in Ordinary Differential Equations},
  year      = {1981},
  address   = {Beijing},
  isbn      = {9781124095455},
  language  = {in Chinese},
}

@Article{Anson2019,
  author    = {Timothy Anson and Eugeny Babichev and Christos Charmousis and Shikhgasan Ramazanov},
  journal   = {Journal of Cosmology and Astroparticle Physics},
  title     = {{Cosmological instability of scalar-Gauss-Bonnet theories exhibiting scalarization}},
  year      = {2019},
  issn      = {1475-7516},
  month     = jun,
  number    = {06},
  pages     = {023--023},
  volume    = {2019},
  doi       = {10.1088/1475-7516/2019/06/023},
  publisher = {IOP Publishing},
}

@Article{Andreou2019,
  author    = {Nikolas Andreou and Nicola Franchini and Giulia Ventagli and Thomas P. Sotiriou},
  journal   = {Physical Review D},
  title     = {Spontaneous scalarization in generalized scalar-tensor theory},
  year      = {2019},
  issn      = {2470-0029},
  month     = jun,
  number    = {12},
  pages     = {124022},
  volume    = {99},
  doi       = {10.1103/physrevd.99.124022},
  publisher = {American Physical Society (APS)},
}

@Article{Bakopoulos2019,
  author    = {Athanasios Bakopoulos and Georgios Antoniou and Panagiota Kanti},
  journal   = {Physical Review D},
  title     = {{Novel black-hole solutions in Einstein-scalar-Gauss-Bonnet theories with a cosmological constant}},
  year      = {2019},
  issn      = {2470-0029},
  month     = mar,
  number    = {6},
  pages     = {064003},
  volume    = {99},
  doi       = {10.1103/physrevd.99.064003},
  publisher = {American Physical Society (APS)},
}

@Article{Doneva2018a,
  author    = {Daniela Doneva and Stoytcho Yazadjiev},
  journal   = {Journal of Cosmology and Astroparticle Physics},
  title     = {{Neutron star solutions with curvature induced scalarization in the extended Gauss-Bonnet scalar-tensor theories}},
  year      = {2018},
  issn      = {1475-7516},
  month     = apr,
  number    = {04},
  pages     = {011--011},
  volume    = {2018},
  doi       = {10.1088/1475-7516/2018/04/011},
  publisher = {IOP Publishing},
}

@Article{Antoniou2018a,
  author    = {Georgios Antoniou and Athanasios Bakopoulos and Panagiota Kanti},
  journal   = {Physical Review D},
  title     = {{Black-hole solutions with scalar hair in Einstein-scalar-Gauss-Bonnet theories}},
  year      = {2018},
  issn      = {2470-0029},
  month     = apr,
  number    = {8},
  pages     = {084037},
  volume    = {97},
  doi       = {10.1103/physrevd.97.084037},
  publisher = {American Physical Society (APS)},
}

@Article{Astefanesei2020,
  author    = {Astefanesei, Dumitru and Herdeiro, Carlos and Oliveira, João and Radu, Eugen},
  journal   = {Journal of High Energy Physics},
  title     = {Higher dimensional black hole scalarization},
  year      = {2020},
  issn      = {1029-8479},
  month     = sep,
  number    = {9},
  volume    = {2020},
  doi       = {10.1007/jhep09(2020)186},
  publisher = {Springer Science and Business Media LLC},
}

@article{Ruth2006,
  title = {Classification of cosmological milestones},
  author = {L. Fern\'andez-Jambrina and Ruth Lazkoz},
  journal = {Phys. Rev. D},
  volume = {74},
  issue = {6},
  pages = {064030},
  numpages = {11},
  year = {2006},
  month = {Sep},
  publisher = {American Physical Society},
  doi = {10.1103/PhysRevD.74.064030},
  url = {https://link.aps.org/doi/10.1103/PhysRevD.74.064030}
}

@Article{Doneva2024,
  author    = {Daniela Doneva and Fethi Mubin Ramazanoğlu and Hector O. Silva and Thomas Sotiriou and Stoytcho Yazadjiev},
  journal   = {Reviews of Modern Physics},
  title     = {{Spontaneous scalarization}},
  year      = {2024},
  issn      = {1539-0756},
  month     = mar,
  number    = {1},
  pages     = {015004},
  volume    = {96},
  doi       = {10.1103/revmodphys.96.015004},
  publisher = {American Physical Society (APS)},
}

@Article{Minamitsuji2019,
  author    = {Masato Minamitsuji and Taishi Ikeda
},
  journal   = {Physical Review D},
  title     = {{Spontaneous scalarization of black holes in the Horndeski theory}},
  year      = {2019},
  issn      = {2470-0029},
  month     = may,
  number    = {10},
  pages     = {104069},
  volume    = {99},
  doi       = {10.1103/physrevd.99.104069},
  publisher = {American Physical Society (APS)},
}

@Article{BlazquezSalcedo2018,
  author    = {Jose Luis Blázquez-Salcedo and  Daniela D. Doneva and  Jutta Kunz and  Stoytcho S. Yazadjiev},
  journal   = {Physical Review D},
  title     = {{Radial perturbations of the scalarized Einstein-Gauss-Bonnet black holes}},
  year      = {2018},
  issn      = {2470-0029},
  month     = oct,
  number    = {8},
  pages     = {084011},
  volume    = {98},
  doi       = {10.1103/physrevd.98.084011},
  publisher = {American Physical Society (APS)},
}

@Article{Dima2020,
  author    = {Alexandru Dima and Enrico Barausse and  Nicola Franchini and  Thomas P. Sotiriou},
  journal   = {Physical Review Letters},
  title     = {Spin-Induced Black Hole Spontaneous Scalarization},
  year      = {2020},
  issn      = {1079-7114},
  month     = dec,
  number    = {23},
  pages     = {231101},
  volume    = {125},
  doi       = {10.1103/physrevlett.125.231101},
  publisher = {American Physical Society (APS)},
}

@Article{Herdeiro2021,
  author    = {Carlos A.R. Herdeiro and  Eugen Radu and Hector O. Silva and Thomas P. Sotiriou and  Nicolás Yunes},
  journal   = {Physical Review Letters},
  title     = {Spin-Induced Scalarized Black Holes},
  year      = {2021},
  month     = {Jan},
  pages     = {011103},
  volume    = {126},
  doi       = {10.1103/PhysRevLett.126.011103},
  issue     = {1},
  numpages  = {6},
  publisher = {American Physical Society},
  url       = {https://link.aps.org/doi/10.1103/PhysRevLett.126.011103},
}

@Article{Berti2021,
  author    = {Emanuele Berti and Lucas Gardai Collodel and Burkhard Kleihaus and Jutta Kunz},
  journal   = {Physical Review Letters},
  title     = {Spin-Induced Black Hole Scalarization in{ Einstein-Scalar-Gauss-Bonnet} Theory},
  year      = {2021},
  issn      = {1079-7114},
  month     = jan,
  number    = {1},
  pages     = {011104},
  volume    = {126},
  doi       = {10.1103/physrevlett.126.011104},
  publisher = {American Physical Society (APS)},
}

@Article{ZH.STEFANOV2008,
  author    = {Ivan Zhivkov Stefanov and Stoytcho Yazadjiev and  Michail D. Todorov},
  journal   = {Modern Physics Letters A},
  title     = {{Phases of 4d scalar–tensor black holes coupled to Born–Infeld nonlinear electrodynamics}},
  year      = {2008},
  issn      = {1793-6632},
  month     = nov,
  number    = {34},
  pages     = {2915--2931},
  volume    = {23},
  doi       = {10.1142/s0217732308028351},
  publisher = {World Scientific Pub Co Pte Lt},
}

@Article{Herdeiro2018,
  author    = {Carlos A. R. Herdeiro and  Eugen Radu and Nicolas Sanchis-Gual and  José A. Font},
  journal   = {Phys. Rev. Lett.},
  title     = {Spontaneous Scalarization of Charged Black Holes},
  year      = {2018},
  month     = {Sep},
  pages     = {101102},
  volume    = {121},
  doi       = {10.1103/PhysRevLett.121.101102},
  issue     = {10},
  numpages  = {6},
  publisher = {American Physical Society},
  url       = {https://link.aps.org/doi/10.1103/PhysRevLett.121.101102},
}

@Article{Cardoso2020,
  author    = {Vitor Cardoso and  Arianna Foschi and Miguel Zilhão},
  journal   = {Phys. Rev. Lett.},
  title     = {Collective Scalarization or Tachyonization: When Averaging Fails},
  year      = {2020},
  month     = {Jun},
  pages     = {221104},
  volume    = {124},
  doi       = {10.1103/PhysRevLett.124.221104},
  issue     = {22},
  numpages  = {6},
  publisher = {American Physical Society},
  url       = {https://link.aps.org/doi/10.1103/PhysRevLett.124.221104},
}

@Article{Cardoso2013,
  author    = {Vitor Cardoso and Isabella P Carucci and Paolo Pani and Thomas Sotiriou},
  journal   = {Phys. Rev. Lett.},
  title     = {Black Holes with Surrounding Matter in Scalar-Tensor Theories},
  year      = {2013},
  month     = {Sep},
  pages     = {111101},
  volume    = {111},
  doi       = {10.1103/PhysRevLett.111.111101},
  issue     = {11},
  numpages  = {5},
  publisher = {American Physical Society},
  url       = {https://link.aps.org/doi/10.1103/PhysRevLett.111.111101},
}

@Article{Antoniadis1994,
  author    = {Ignatios Antoniadis and John Rizos and Kyriakos Tamvakis},
  journal   = {Nuclear Physics B},
  title     = {{Singularity-free cosmological solutions of the superstring effective action}},
  year      = {1994},
  issn      = {0550-3213},
  month     = mar,
  number    = {2},
  pages     = {497--514},
  volume    = {415},
  doi       = {10.1016/0550-3213(94)90120-1},
  publisher = {Elsevier BV},
}

@Article{Easther1996,
  author    = {Richard Easther and Kei-ichi Maeda},
  journal   = {Physical Review D},
  title     = {One-loop superstring cosmology and the nonsingular universe},
  year      = {1996},
  issn      = {1089-4918},
  month     = dec,
  number    = {12},
  pages     = {7252--7260},
  volume    = {54},
  doi       = {10.1103/physrevd.54.7252},
  publisher = {American Physical Society (APS)},
}

@Article{Hartle1983,
  author    = {James B. Hartle and Stephen W. Hawking},
  journal   = {Physical Review D},
  title     = {Wave function of the Universe},
  year      = {1983},
  issn      = {0556-2821},
  month     = dec,
  number    = {12},
  pages     = {2960--2975},
  volume    = {28},
  doi       = {10.1103/physrevd.28.2960},
  publisher = {American Physical Society (APS)},
}

\end{document}